\documentclass[11pt]{amsart}

\usepackage{amsthm,amsmath,amssymb,amsfonts,amscd,amsgen,upref,esint,nicefrac}

\usepackage[margin=1in]{geometry}

\theoremstyle{plain}
\newtheorem{cor}{Corollary}
\newtheorem{lem}[cor]{Lemma}
\newtheorem{prop}[cor]{Proposition}

\newtheorem{thm}[cor]{Theorem}
\newtheorem*{thm*}{Theorem}
\theoremstyle{definition}
\newtheorem{definition}[cor]{Definition}
\newtheorem{remark}[cor]{Remark}
\newtheorem*{remark*}{Remark}

\numberwithin{cor}{section}
\numberwithin{equation}{section}

\DeclareMathOperator{\C}{C}

\DeclareMathOperator{\BUC}{BUC}

\DeclareMathOperator{\sgn}{sgn}

\newcommand{\abs}[1]{\left|#1\right|}
\newcommand{\norm}[1]{\left\|#1\right\|}
\providecommand{\ud}[1]{\, \mathrm{d} #1}
\providecommand{\dx}{\ud{x}}
\providecommand{\dy}{\ud{y}}
\providecommand{\dxi}{\ud \xi}
\providecommand{\deta}{\ud{\eta}}
\providecommand{\dr}{\ud{r}}

\providecommand{\dxp}{\ud{x'}}
\providecommand{\dxip}{\ud{\xi'}}

\providecommand{\ds}{\ud{s}}

\providecommand{\dz}{\ud{z}}
\providecommand{\dd}{\ud}
\providecommand{\ve}{\varepsilon}

\def\XXint#1#2#3{{\setbox0=\hbox{$#1{#2#3}{\int}$ }
\vcenter{\hbox{$#2#3$ }}\kern-.6\wd0}}

\title[stochastic nonlinear diffusion equations]{Path-by-path well-posedness of nonlinear diffusion equations with multiplicative noise}

\author{Benjamin Fehrman, Benjamin Gess}

\date{\today}

\begin{document}

\begin{abstract}  We prove the path-by-path well-posedness of stochastic porous media and fast diffusion equations driven by linear, multiplicative noise. As a consequence, we obtain the existence of a random dynamical system. This solves an open problem raised in [Barbu, R\"ockner; JDE, 2011], [Barbu, R\"ockner; JDE, 2018], and [Gess, AoP, 2014].
\end{abstract}

\maketitle

\section{Introduction}

In this paper, we consider stochastic porous media and fast diffusion equations with linear multiplicative noise of the form 
\begin{equation}
\left\{ \begin{array}{ll}
du=\Delta\left(u\abs{u}^{m-1}\right)+\sum_{k=1}^{n}f_{k}(x)u\circ\dz^{k}_t & \textrm{in}\;\;U\times(0,\infty),\\
u=u_{0} & \textrm{on}\;\;U\times\{0\},\\
u=0 & \textrm{on}\;\;\partial U\times(0,\infty),
\end{array}\right.\label{intro_eq}
\end{equation}
with nonnegative initial data $u_{0}\in L^{2}(U)$, diffusion exponent $m\in(0,\infty)$, domain $U\subset\mathbb{R}^{d}$ smooth and bounded, $d\ge 1$, driving noise $z=(z^{1},\ldots,z^{n})\in C^{0}(\mathbb{R}_{+};\mathbb{R}^{n})$, and coefficients $\{f_{k}\}_{k=1}^{n}\in\C_{b}^{2}(U;\mathbb{R}^{n})$ for some $n\in\mathbb{N}$.

Stochastic nonlinear diffusion equations of the type \eqref{intro_eq} have a rich and rapidly developing history (cf.~e.g.~Pardoux \cite{P75}, Krylov and Rozovskii \cite{KR79}, Pr\'evot and R\"ockner \cite{PR07}, Barbu and R\"ockner \cite{BR15}, Barbu, Da Prato, and R\"ockner \cite{BDPR16}, Ren, R\"ockner, and Wang \cite{RenRocWan2007}, the second author and R\"ockner \cite{GR}). Their probabilistic well-posedness, in the sense of pathwise uniqueness and probabilistically weak uniqueness, is thus rather well understood. In contrast, in the fast diffusion case $m\in(0,1)$, the path-by-path well-posedness of \eqref{intro_eq} and the related problem of the existence of a corresponding stochastic flow and random dynamical system have proven to be notoriously difficult and have been posed as open questions in several works, e.g.\ Barbu and R\"ockner \cite{BR11b,BR15,BR17} and the second author \cite{Gess1,G15}. In addition, even in the porous medium case $m\in(1,\infty)$, for general initial data $u_0\in L^2(U)$, path-by-path solutions could only be obtained in a limiting sense in \cite{BR11b,BR17,Gess1}, thus lacking a characterization in terms of (generalized) solutions to \eqref{intro_eq}. 

The theory of random dynamical systems is the attempt of joining methods from stochastic analysis and dynamical systems, in order to make accessible the qualitative analysis of (semi-)flows of solutions to stochastic (partial) differential equations to tools from dynamical systems theory. This attempt is motivated and driven by the success of the respective theory in the analysis of deterministic (partial) differential equations, see, for example, Temam \cite{T97}. For instance, this allows the application of the multiplicative ergodic theorem, leading to the concepts of Lyapunov exponents and invariant manifolds. In contrast to deterministic (partial) differential equations, the (semi-)flow property for solutions to stochastic  (partial) differential equations is not a simple consequence of the uniqueness of solutions. While in finite dimensions this issue can be handled with methods based on Kolmogorov's continuity theorem, the infinite dimensional case lacks a generic treatment and counter-examples are known. The construction of (semi-)flows of solutions thus is a key obstacle to the dynamical systems approach to stochastic PDE. The present work contributes to this challenging problem by expanding the class of nonlinear diffusion equations with linear multiplicative noise for which a (semi-)flow of solutions can be constructed, and by introducing a unified construction of random dynamical systems for this class of stochastic PDE.

The main result of this work is the path-by-path well-posedness of \eqref{intro_eq} in the full range $m\in(0,\infty)$ which implies the existence of a corresponding random dynamical system.  This is the first path-by-path well-posedness result in the fast diffusion range $m\in(0,1)$, and our results and methods provide a unifying approach the whole range $m\in(0,\infty)$.

\begin{thm*}[Theorem \ref{thm_unique} below] Let $u_{0}^{1},u_{0}^{2}\in L^{2}(U)$ be nonnegative and $u^{1}, u^{2}$ be corresponding pathwise kinetic solutions to \eqref{intro_eq}.  Then, for each $T>0$, there exists $C=C(T,z)>0$ such that
\[
\norm{\left(u^{1}(t)-u^{2}(t)\right)\varphi}_{L^\infty([0,T];L^{1}(U))}\leq C\norm{\left(u_{0}^{1}-u_{0}^{2}\right)\varphi}_{L^{1}(U)},
\]
where $\varphi$ is the solution to \eqref{eq:intro_vp} below. In particular, pathwise kinetic solutions are unique. \end{thm*}

\begin{thm*}[Theorem \ref{thm_exist} below] Let $u_{0}\in L^{2}(U)$. There exists a pathwise kinetic solution of \eqref{intro_eq} with initial data $u_{0}$. Furthermore, for each $T>0$, for $C=C(m,U,T,z)>0$, 
\[
\norm{u}_{L^{\infty}([0,T];L^{2}(U))}^{2}+\norm{u^{\left[\frac{m+1}{2}\right]}}_{L^{2}([0,T];H_{0}^{1}(U))}^{2}\leq C\norm{u_{0}}_{L^{2}(U)}^{2},
\]
and, if $u_{0}$ is nonnegative, then so is $u$.
\end{thm*}

As discussed above, the task of proving path-by-path well-posedness for \eqref{intro_eq} is partially motivated by the question of the existence of a random dynamical system associated to \eqref{intro_eq}. A general introduction to the theory of random dynamical systems can be found in Arnold \cite{A98} and Flandoli \cite{Flandoli}.  In combination, the results above prove the existence of a random dynamical system for \eqref{intro_eq} on the space $L_{+}^{2}(U)$. For simplicity, we specialize the statement to the case of fractional Brownian motion.

\begin{thm*} Suppose that $t\in[0,\infty)\mapsto z_{t}(\omega)$ is given as the sample paths of a fractional Brownian motion with Hurst parameter $H\in(0,1)$. Then, the pathwise kinetic solutions to \eqref{intro_eq} define a random dynamical system on $L_{+}^{2}(U)$. \end{thm*}

\begin{remark*}  The $L^2$-integrability of the initial data is assumed for simplicity.  The results of this paper can be extended to nonnegative initial data in $L^1(U)$ at the cost of additional technicalities.  In particular, the definition of a pathwise kinetic solution needs to be modified, since the \emph{entropy} and \emph{parabolic defect measures} will no longer be globally integrable (cf.\ Definition~\ref{sol_def} below).  The proof of uniqueness and the stable estimates would also need to be localized in order to account for the lack of integrability. \end{remark*}

\subsection*{Aspects of the proof and motivation of the kinetic formulation}  The arguments of this work rely on the kinetic formulation of \eqref{intro_eq}, introduced by Chen and Perthame \cite{ChenPerthame}. Motivated by the theory of stochastic viscosity solutions for fully-nonlinear second-order stochastic partial differential equations by Lions and Souganidis \cite{LSstoch5,LSstoch4,LSstoch3,LSstoch2,LSstoch1}, and the work of Lions, Perthame and Souganidis \cite{LPS1,LPS}, the second author and Souganidis \cite{GessSouganidis,GessSouganidis1,GessSouganidis2}, and the two authors \cite{FehrmanGess} on stochastic scalar conservation laws, this gives rise to the notion of a \emph{pathwise kinetic solution} (cf. Definition~\ref{sol_def} below).

In the following we present an informal sketch of the proof of uniqueness of solutions that motivates the introduction of the kinetic formulation of \eqref{intro_eq}. It is at this point where the methods of the present work deviate from the previous, un-successful attempts \cite{BR11b,BR15,BR17,Gess1,G15} and the kinetic formulation proves to be essential in the fast diffusion case $m\in(0,1)$.

Previous methods to treat \eqref{intro_eq} have exploited the multiplicative structure of the noise, defining the weight
\begin{equation}\label{intro_weight}v(x,t)=\exp\left(-\sum_{k=1}^n\int_0^tf_k(x)\dd z^k_s\right),\end{equation}
to transform \eqref{intro_eq} to the form
\begin{equation}\label{eqn:intro-random-PDE}
  \partial_t(vu)=v(x,t)\Delta u^{[m]}.
\end{equation}
While this transformation removes the stochastic integral, the presence of the weight  $v$ complicates the application of the above methods and, in particular, the contractivity and the uniqueness of solutions to \eqref{intro_eq} with respect to the standard $L^1$-norm seem unclear. 

To overcome these difficulties, motivated by the work of the second author \cite{Gess1} in the porous media regime, we introduce a weighted $L^{1}$-norm with positive weight $\varphi\in\C^{\infty}(U)$ defined by
\begin{equation}\label{eq:intro_vp}
  \left\{ \begin{array}{ll}
  \Delta\varphi=-1 & \textrm{in}\;\;U,\\
  \varphi=0 & \textrm{on}\;\;\partial U.
  \end{array}\right.  
\end{equation}
The introduction of $v$ and $\varphi$ yields a formal proof of uniqueness:  Let $\abs{\cdot}^\ve$ denote a smooth approximation of the absolute value $\abs{\cdot}$, with derivative $\sgn^\ve$ and second-derivative $\delta^\ve_0$ approximating the sign function $\sgn$ and delta distribution $\delta_0$ respectively.  Then, 
$$
\partial_t\int_U \abs{u_1-u_2}^\ve v\varphi = \int_U\sgn^\ve (u_1-u_2)\Delta\left(u_1^{[m]}-u_2^{[m]}\right)v\varphi
$$
and, after integrating by parts,
\begin{equation}\label{intro_mot_1}\begin{aligned}
& \partial_t\int_U \abs{u_1-u_2}^\ve v\varphi = -\int_U\delta^\ve_0\left(u_1-u_2\right)\nabla\left(u_1-u_2\right)\cdot \nabla\left(u_1^{[m]}-u_2^{[m]}\right)v\varphi
\\ & \quad +\int_U\delta^\ve_0\left(u_1-u_2\right)\left(u^{[m]}_1-u^{[m]}_2\right)\nabla\left(u_1-u_2\right)\cdot\nabla(v\varphi)
\\ & \quad +\int_U\sgn^\ve\left(u_1-u_2\right)\left(u^{[m]}_1-u^{[m]}_2\right)\Delta(v\varphi).
\end{aligned}\end{equation}
Since the definitions of $v$ and $\varphi$ imply that for small times $\Delta(v\varphi)\leq 0$, the final term on the righthand side of \eqref{intro_mot_1} is for small times nonpositive in the $\ve\rightarrow 0$ limit.  The second term on the righthand side of \eqref{intro_mot_1} formally vanishes in the $\ve\rightarrow 0$ limit but, in the case $m\in(0,1)$, even this statement requires a more detailed analysis due to the fact that, for small values of $u_1$ and $u_2$,
$$\delta^\ve_0\left(u_1-u_2\right)\left(u^{[m]}_1-u^{[m]}_2\right)\simeq \ve^{-1}\mathbf{1}_{\{\abs{u_1-u_2} \le \ve\}}\abs{u_1-u_2}^m\simeq \ve^{m-1}\mathbf{1}_{\{\abs{u_1-u_2} \le \ve\}}.$$
We will focus for now on the first term on the righthand side of \eqref{intro_mot_1}.  We have formally that
\begin{equation}\label{intro_mot_3}\begin{aligned}
& \nabla\left(u_1-u_2\right)\cdot \nabla\left(u_1^{[m]}-u_2^{[m]}\right) = -m\left(\abs{u_1}^{\frac{m-1}{2}}-\abs{u_2}^{\frac{m-1}{2}}\right)^2\nabla u_1\cdot\nabla u_2
\\ & \quad + m\abs{u_1}^{m-1}\abs{\nabla u_1}^2+m\abs{u_2}^{m-1}\abs{\nabla u_2}^2 - 2m\abs{u_1}^{\frac{m-1}{2}}\abs{u_2}^\frac{m-1}{2}\nabla u_1\cdot\nabla u_2.
\end{aligned}\end{equation}
H\"older's and Young's inequalities prove the final line of \eqref{intro_mot_3} is nonnegative.  So, by the above, we expect to have that
\begin{equation}\label{intro_mot_4}\partial_t\int_U \abs{u_1-u_2}^\ve v\varphi \leq \int_U\delta^\ve_0\left(u_1-u_2\right)m\left(\abs{u_1}^{\frac{m-1}{2}}-\abs{u_2}^{\frac{m-1}{2}}\right)^2\nabla u_1\cdot\nabla u_2v\varphi.\end{equation}
If $m\in[3,\infty)$ or $m=1$ the righthand side vanishes in the $\ve\to 0$ limit, due to the local Lipshitz continuity of $\xi\mapsto\xi^\frac{m-1}{2}$. In contrast, this approach fails in the regime  $m\in(0,1)\cup(1,2)$.  If $m\in(1,2)$, the local H\"older continuity of the map $\xi\mapsto\xi^\frac{m-1}{2}$ proves at best that, for small value of $u_1$, $u_2$,
$$\delta^\ve_0\left(u_1-u_2\right)m\left(\abs{u_1}^{\frac{m-1}{2}}-\abs{u_2}^{\frac{m-1}{2}}\right)^2\simeq \ve^{m-2}\mathbf{1}_{\{\abs{u_1-u_2}\leq\ve\}},$$
and if $m\in(0,1)$ the map $\xi\mapsto\xi^\frac{m-1}{2}$ is singular near zero.  At this point it becomes evident why the previously developed methods in \cite{BR11b,BR15,BR17,Gess1,G15} failed in the case of the fast diffusion equation.

In the present work, this issue is resolved by exploiting the kinetic formulation of \eqref{eqn:intro-random-PDE}. By decomposing the random PDE \eqref{eqn:intro-random-PDE} according to the value of the solution $u$, the kinetic formulation allows to separate regions where the solutions are bounded away from zero, and therefore the nonlinearity $\xi\mapsto|\xi|^{\frac{m-1}{2}}$ is locally Lipschitz, and regions where the solutions are near zero.  The treatment of the latter region relies on fine estimates of commutator errors, which in turn build upon new and sharp regularity estimates of the type
$$ \nabla u^\frac{m}{2} \in L^2([0,T];L^2_{loc}(U)).$$
These new regularity properties and their usage in the proof of uniqueness via the kinetic formulation and commutator techniques are the basis for solving the previously inaccessible fast diffusion case $m\in(0,1)$.

We emphasize again that, in contrast to the case of conservative stochastic PDE treated in \cite{FehrmanGess}, the contractivity of the $L^1$-norm for \eqref{eqn:intro-random-PDE} is not expected.  This is due to the fact that the weight \eqref{intro_weight} appears on the level of the stochastic characteristics \eqref{c_forward_eq} below, which do not preserve the underlying Lebesgue measure as they do in \cite{FehrmanGess}.  Furthermore, the appearance of the weight $v$ in \eqref{eqn:intro-random-PDE} and the inclusion of the weight $\varphi$ in \eqref{eq:intro_vp} lead to additional commutator errors  in the proof of uniqueness (cf.\ e.g.\ the introduction of the cutoff \eqref{tu_1}, Steps 6-8 in the proof of Theorem~\ref{thm_unique}, and Lemma~\ref{lem_bdry} below) which are not present in \cite{FehrmanGess}, and which require careful analysis based on new regularity estimates. Furthermore, the boundary data in \eqref{intro_eq} leads to additional boundary layer errors that need to be treated in the proof of existence and uniqueness, which do not appear in the periodic setting of \cite[Theorem~4.1]{FehrmanGess}. The treatment of these terms and the proof of existence and uniqueness require new estimates based on $L^p$-norms weighted by $\varphi$, since the estimates used in \cite{FehrmanGess} fail near the boundary.

\subsection*{Notation} For $u\in \mathbb{R}$, we set $u^{[m]}:=|u|^{m-1}u$. Further, for $p\ge 1$ we let $L_{+}^{p}(U)$ be the space of a.e.\ nonnegative $L^{p}(U)$-functions. For simplicity we will sometimes use the convention $f_t(x,v):=f(t,x,v)$ for functions or measures $f$.  The notation will frequently omit the integration variables.  For example, for a measureable function $f:\mathbb{R}^d\rightarrow\mathbb{R}$, we will write
$$\int_{\mathbb{R}^d}f=\int_{\mathbb{R}^d}f(x)\dx.$$

\section{Definition and motivation of pathwise kinetic solutions}

In this section, we begin by considering a smooth, elliptic perturbation of \eqref{intro_eq} that is classically well-posed.  We will then derive a formulation of the equation that is well-defined for singular driving signals after passing to the kinetic form, where the noise enters as a linear transport, and passing to the limit with respect to the regularization.

For each $\ve\in(0,1)$, let $\rho^\ve_1$ denote a standard one-dimensional convolution kernel of scale $\ve$.   Define the smooth path
\begin{equation}\label{convolve_noise}z^\ve_t=(z^{1,\ve}_t,\ldots,z^{n,\ve}_t):=\int_\mathbb{R}z_{(s\vee0)}\rho^\ve_1(s-t)\ds.\end{equation}
We will write $\dot{z}^\ve$ for the time-derivative of the path $z^\ve$ and, for each $\eta\in(0,1)$ and $\ve\in(0,1)$, we consider the equation
\begin{equation}\label{reg_eq}\left\{\begin{array}{ll} \partial_tu^{\eta,\ve}=\Delta \left(u^{\eta,\ve}\right)^{[m]}+\eta\Delta u^{\eta,\ve}+\sum_{k=1}^nf_k(x)u^{\eta,\ve}\dot{z}^{k,\ve}_t & \textrm{in}\;\; U\times(0,\infty), \\ u^{\eta,\ve}=u_0 & \textrm{on}\;\; U \times\{0\}, \\ u^{\eta,\ve}=0 & \textrm{on}\;\;\partial U\times(0,\infty).\end{array}\right.\end{equation}
The following proposition establishes the well-posedness of \eqref{reg_eq}.  The proof is omitted, since it is only a small modification of \cite[Proposition~6.1]{FehrmanGess}.

\begin{prop}\label{reg_exists}  Let $\eta\in(0,1)$, $\ve\in(0,1)$, and $u_0\in L^2(U)$ be arbitrary.  There exists a unique solution $u^{\eta,\ve}$ of \eqref{reg_eq} satisfying, for each $T>0$, for $C=C(m,U,T,\ve)>0$,
$$\norm{u^{\eta,\ve}}^2_{L^\infty([0,T];L^2(U))}\leq C\norm{u_0}^2_{L^2(U)},$$
and
$$ \norm{\eta^\frac{1}{2}\nabla u^{\eta,\ve}}^2_{L^2([0,T];L^2(U;\mathbb{R}^d))} +\norm{\nabla \left(u^{\eta,\ve}\right)^{\left[\frac{m+1}{2}\right]}}^2_{L^2([0,T];L^2(U;\mathbb{R}^d))}\leq C\norm{u_0}_{L^2(U)}^2.$$
\end{prop}

We will now derive the kinetic formulation of \eqref{reg_eq}.  The arguments are similar to \cite[Section~6]{FehrmanGess}, and a general introduction to the theory of kinetic solutions can be found in Perthame \cite{Perthame} and \cite{ChenPerthame}.  The kinetic function $\overline{\chi}:\mathbb{R}^2\rightarrow\{-1,0,1\}$ is defined by $\overline{\chi}(s,\xi)=\mathbf{1}_{\{0<\xi<s\}}-\mathbf{1}_{\{s<\xi<0\}}$.  For each $\eta\in(0,1)$ and $\ve\in(0,1)$, for the solution $u^{\eta,\ve}$ from Proposition~\ref{reg_exists}, we consider the composition
\begin{equation}\label{char_kinetic} \chi^{\eta,\ve}(x,\xi,t)=\overline{\chi}(u^{\eta,\ve}(x,t),\xi).\end{equation}
The kinetic function \eqref{char_kinetic} formally satisfies
\begin{equation}\label{char_kinetic_eq}\partial_t\chi^{\eta,\ve}= m\abs{\xi}^{m-1}\Delta \chi^{\eta,\ve}+\eta\Delta\chi^{\eta,\ve}-\partial_\xi \chi^{\eta,\ve}\sum_{k=1}^n \xi f_k(x)\dot{z}^\ve_t+\partial_\xi\left(p^{\eta,\ve}+q^{\eta,\ve}\right),\end{equation}
for the entropy defect measure
$$p^{\eta,\ve}(x,\xi,t):=\delta_0(\xi-u^{\eta,\ve}(x,t))\abs{\eta^\frac{1}{2}\nabla u^{\eta,\ve}(x,t)}^2,$$
and for the parabolic defect measure
$$ q^{\eta,\ve}(x,\xi,t):=\delta_0(\xi-u^{\eta,\ve}(x,t))\frac{4m}{(m+1)^2}\abs{\nabla \left(u^{\eta,\ve}\right)^{\left[\frac{m+1}{2}\right]}(x,t)}^2,$$
where $\delta_0$ denotes the one-dimensional Dirac distribution at the origin.  A small modification of \cite[Proposition~6.2]{FehrmanGess} proves that, on the level of distributions, this is indeed the case.

\begin{prop}\label{char_keq}  Let $\eta\in(0,1)$, $\ve\in(0,1)$, and $u_0\in L^2(U)$ be arbitrary.  Let $u^{\eta,\ve}$ denote the solution of \eqref{reg_eq} from Proposition~\ref{reg_exists}, and let $\chi^{\eta,\ve}$ denote its corresponding kinetic function defined in \eqref{char_kinetic}.  Then, $\chi^{\eta,\ve}$ is a distributional solution of \eqref{char_kinetic_eq} in the sense that, for each $\psi\in\C^\infty_c(U\times\mathbb{R}\times[0,\infty))$, for each $0\leq s<t<\infty$, 
\begin{equation}\label{char_keq_1}\begin{aligned} & \left.\int_{\mathbb{R}^{d+1}}\chi^{\eta,\ve}(x,\xi,r)\psi(x,\xi,r)\right|_{r=s}^t
\\ & = \int_s^t\int_\mathbb{R}\int_Um\abs{\xi}^{m-1}\chi^{\eta,\ve}\Delta \psi+\eta\chi^{\eta,\ve}\Delta\psi + \sum_{k=1}^n\int_s^t\int_{\mathbb{R}}\int_U \chi^{\eta,\ve}\partial_\xi\left(\psi \xi f_k(x)\dot{z}^{k,\ve}_r\right) \\ & \quad +\int_s^t\int_\mathbb{R}\int_U \chi^{\eta,\ve}\partial_t\psi  - \int_s^t\int_{\mathbb{R}}\int_U \left(p^{\eta,\ve}+q^{\eta,\ve}\right)\partial_\xi\psi.\end{aligned}\end{equation}
\end{prop}

We will now obtain an interpretation of \eqref{char_keq_1} that is well-defined for rough driving signals.  The idea, which is motivated by the theory of stochastic viscosity solutions \cite{LSstoch5,LSstoch4,LSstoch3,LSstoch2,LSstoch1}, is to consider a class of test function that are transported by the stochastic characteristic defining the linear transport term appearing in \eqref{char_kinetic_eq}.  In comparison to \cite{FehrmanGess,GessSouganidis,GessSouganidis2}, however, the application of these ideas is complicated by the fact that the flow of the characteristics \eqref{c_forward} below does not preserve the Lebesgue measure.

The forward characteristic $\Xi^{\ve,x,\xi}_{s,t}$ beginning from time $s\geq 0$ and $(x,\xi)\in U\times\mathbb{R}^d$ is the solution to
\begin{equation}\label{c_forward_eq} \left\{\begin{array}{ll} \dot{\Xi}^{\ve,x,\xi}_{s,t}=\sum_{k=1}^n\Xi^{\ve,x,\xi}_{s,t}f_k(x)\dot{z}^{k,\ve}_{s,t} & \textrm{in}\;\;(s,\infty), \\ \Xi^{\ve,x,\xi}_{s,s}=\xi, & \end{array}\right.\end{equation}
for $z^{k,\ve}_{s,t}:=z^{k,\ve}_t-z^{k,\ve}_s.$  The linearity implies that \eqref{c_forward_eq} has the explicit solution
\begin{equation}\label{c_forward}\Xi^{\ve,x,\xi}_{s,t}=\xi\exp\left(\sum_{k=1}^nf_k(x)z^{k,\ve}_{s,t}\right).\end{equation}
The inverse characteristic is constructed explicitly and is defined for each $t\geq 0$ and $s\in[0,t]$ by
\begin{equation}\label{c_backward} \Pi^{\ve,x,\xi}_{t,t-s}:=\xi \exp\left(-\sum_{k=1}^nf_k(x)z^{k,\ve}_{s,t}\right)=\xi \exp\left(\sum_{k=1}^nf_k(x)z^{k,\ve}_{t,s}\right).\end{equation}
Indeed, it follows by definition that the characteristics \eqref{c_forward} and \eqref{c_backward} are mutually inverse in the sense that, for each $0\leq s \leq t<\infty$, for each $(x,\xi)\in U\times\mathbb{R}$,
\begin{equation}\label{c_inverse} \Xi^{\ve,x,\Pi^{\ve,x,\xi}_{t,t-s}}_{s,t}=\xi\;\;\textrm{and}\;\; \Pi^{\ve,x,\Xi^{\ve,x,\xi}_{t-s,s}}_{t,s}=\xi.\end{equation}

We aim to study the equation satisfied by the transported kinetic function $\chi^{\eta,\ve}(x,\Xi^{\ve,x,\xi}_{s,t},t)$.  Technically, this is achieved by testing \eqref{char_keq_1} with test functions transported along the inverse characteristics \eqref{c_backward}.  This transport and the corresponding change of measure are described, for an arbitrary $\rho_0\in \C^\infty_c(U\times\mathbb{R})$ and $s\geq 0$, by the solution
\begin{equation}\label{c_flow} \left\{\begin{array}{ll} \partial_t \rho^\ve_{s,t}=-\partial_\xi\left(\rho^\ve_{s,t}\xi \sum_{k=1}^nf_k(x)\dot{z}^{k,\ve}_t\right) & \textrm{in}\;\;U\times\mathbb{R}\times(s,\infty), \\ \rho^\ve_{s,s}=\rho_0 & \textrm{on}\;\;U\times\mathbb{R}\times\{s\},\end{array}\right.\end{equation}
which is the conservative equation dual to the linear transport term in \eqref{char_kinetic_eq}.  The weight, defined for each $0\leq s<t<\infty$ and $x\in U$,
\begin{equation}\label{c_weight} v^\ve_{s,t}(x):=\left(\partial_\xi \Xi^{\ve,x,\xi}_{s,t}\right)^{-1}=\exp\left(-\sum_{k=1}^nf_k(x)z^{k,\ve}_{s,t}\right),\end{equation}
defines the change of measure introduced by transport along the characteristics.  That is, for each $0\leq s\leq t<\infty$, for each $\psi\in\C^\infty_c(U\times\mathbb{R})$,
$$\int_\mathbb{R}\int_U\psi(x,\Xi^{\ve,x,\xi}_{s,t})\dx\dxi=\int_\mathbb{R}\int_U\psi(x,\xi) v^\ve_{s,t}(x)\dx\dxi.$$
It is essential to our analysis that $v^{\ve}$ does not depend on the velocity variable $\xi\in\mathbb{R}$.

It follows from \eqref{c_flow} that, for each $s\geq 0$, the product
\begin{equation}\label{c_tilde}\tilde{\rho}^\ve_{s,t}(x,\xi):=\rho^\ve_{s,t}(x,\xi)v^\ve_{s,t}(x)^{-1},\end{equation}
satisfies a pure transport equation
$$\left\{\begin{array}{ll} \partial_t \tilde{\rho}^\ve_{s,t}=-\xi\partial_\xi\tilde{\rho}^\ve_{s,t}\sum_{k=1}^nf_k\dot{z}^\ve_t & \textrm{in}\;\;U\times\mathbb{R}\times(s,\infty), \\ \tilde{\rho}^\ve_{s,s}=\rho_0 & \textrm{on}\;\;U\times\mathbb{R}\times\{s\}.\end{array}\right.$$
Therefore, in view of \eqref{c_backward}, an explicit computation proves that
$$\tilde{\rho}^\ve_{s,t}(x,t)=\rho_0\left(x,\Pi^{\ve,x,\xi}_{t,t-s}\right).$$
Returning to \eqref{c_tilde}, for each $s\geq 0$, it follows that the solution to \eqref{c_flow} is given by
\begin{equation}\label{c_flow_1} \rho^\ve_{s,t}(x,\xi)=\rho_0\left(x,\Pi^{\ve,x,\xi}_{t,t-s}\right)v^\ve_{s,t}(x).\end{equation}

The following proposition proves that the transport of test functions by equation \eqref{c_flow}, in the sense of transport by the inverse characteristics \eqref{c_backward} and weighting by the change of variables factor \eqref{c_weight}, cancels the noise appearing in \eqref{char_keq_1}.  The proof is a small modification of \cite[Proposition~3.3]{FehrmanGess}, and is therefore omitted.

\begin{prop}\label{c_eq}  Let $\eta\in(0,1)$, $\ve\in(0,1)$, and $u_0\in L^2(U)$ be arbitrary.  Let $u^{\eta,\ve}$ denote the solution from Proposition~\ref{reg_exists}, and let $\chi^{\eta,\ve}$ denote the corresponding kinetic function from \eqref{char_kinetic}.  For each $0\leq s<t<\infty$ and $\rho_0\in\C^\infty_c(U\times\mathbb{R})$, for $\rho^\ve_{s,r}$ defined in \eqref{c_flow_1},
\begin{equation}\label{c_eq_1}\left.\int_\mathbb{R}\int_U\chi^{\eta,\ve}\rho^\ve_{s,r}\right|_{r=s}^t =  \int_s^t\int_\mathbb{R}\int_U m\abs{\xi}^{m-1}\chi^{\eta,\ve}\Delta\rho^\ve_{s,r}+\eta\chi^{\eta,\ve}\Delta\rho^\ve_{s,r}- \int_s^t\int_\mathbb{R}\int_U \left(p^{\eta,\ve}+q^{\eta,\ve}\right)\partial_\xi \rho^\ve_{s,r}.\end{equation}
\end{prop}

The essential observation is that equation \eqref{c_eq_1} is well-defined in the singular limit $\ve\rightarrow 0$.  In particular, we have the forward and backward characteristics, defined for each $0\leq s<t<\infty$ and $(x,\xi)\in U\times\mathbb{R}$,
\begin{equation}\label{c_char} \Xi^{x,\xi}_{s,t}:=\xi\exp\left(\sum_{k=1}^n f_k(x) z^k_{s,t}\right)\;\;\textrm{and}\;\;\Pi^{x,\xi}_{t,t-s}:=\xi\exp\left(\sum_{k=1}^nf_k(x)z^k_{t,s}\right),\end{equation}
and the change of variables factor
\begin{equation}\label{c_w} v_{s,t}(x,\xi):=\left(\partial_\xi \Xi^{x,\xi}_{s,t}\right)^{-1}=\exp\left(-\sum_{k=1}^n f_k(x) z^k_{s,t}\right).\end{equation}
For each $\rho_0\in \C^\infty_c(U\times\mathbb{R})$, for each $0\leq s<t<\infty$ and $(x,\xi)\in U\times\mathbb{R}$, we define
\begin{equation}\label{c_f} \rho_{s,t}(x,\xi):=\rho_0\left(x,\Pi^{x,\xi}_{t,t-s}\right)v_{s,t}(x),\end{equation}
which is formally the solution of \eqref{c_flow} driven by singular noise.  We are now prepared to present the definition of a pathwise kinetic solution to \eqref{intro_eq}.

\begin{definition}\label{sol_def} Let $u_0\in L^2(U)$.  A \emph{pathwise kinetic solution} of \eqref{intro_eq} with initial data $u_0$ is a function $u\in L^\infty_{\textrm{loc}}([0,\infty);L^2(U))$ that satisfies the following two properties.

(i)  For each $T>0$, 
$$u^{\left[\frac{m+1}{2}\right]}\in L^2([0,T];H^1_0(U)).$$
In particular, for each $T>0$, the parabolic defect measure
$$q(x,\xi,t):=\frac{4m}{(m+1)^2}\delta_0(\xi-u(x,t))\abs{\nabla u^{\left[\frac{m+1}{2}\right]}}^2,$$
is finite on $U\times\mathbb{R}\times (0,T)$.

(ii)  For the kinetic function $\chi(x,\xi,t):=\overline{\chi}(u(x,t),\xi),$ there exists a nonnegative entropy defect measure $p$ on $U\times\mathbb{R}\times(0,\infty)$, which is finite on $U\times\mathbb{R}\times(0,T)$ for each $T>0$,  and a subset $\mathcal{N}\subset(0,\infty)$ of Lebesgue measure zero such that, for every $s<t\in[0,\infty)\setminus\mathcal{N}$, for every $\rho_0\in\C^\infty_c(U\times\mathbb{R})$, for $\rho_{s,r}$ defined in \eqref{c_f},
\begin{equation}\label{transport_equation}\left.\int_\mathbb{R}\int_U\chi\rho_{s,r}\dx\dxi\right|_{r=s}^t = \int_s^t\int_\mathbb{R}\int_U m\abs{\xi}^{m-1}\chi\Delta \rho_{s,r}\dx\dxi\dr - \int_s^t\int_\mathbb{R}\int_U \left(p+q\right) \partial_\xi \rho_{s,r}\dx\dxi\dr,\end{equation}
where the initial condition is attained in the sense that, when $s=0$,
$$\int_\mathbb{R}\int_U\chi(x,\xi,0)\rho_{0,0}(x,\xi)\dx\dxi=\int_\mathbb{R}\int_U\overline{\chi}(u_0(x),\xi)\rho_0(x,\xi)\dx\dxi.$$
\end{definition}

We observe that property (i) of Definition~\ref{sol_def} implies that the boundary condition is attained in the sense that $u^{\left[\nicefrac{m+1}{2}\right]}$ has a vanishing trace on the boundary.  In the proof of uniqueness, we will require the more standard assumption that $u^{\left[m\right]}$ has vanishing trace on the boundary.  The equivalence of these two conditions is proven in Lemma~\ref{lem_bdry} of the appendix.  Finally, we will use the fact that the solutions satisfy the following integration by parts formula:  for each $\psi\in\C^\infty_c(U\times\mathbb{R}\times[0,\infty))$, for each $t\geq 0$,
\begin{equation}\label{equation_ibp}
\int_0^t\int_\mathbb{R}\int_U \frac{m+1}{2}\abs{\xi}^\frac{m-1}{2}\chi(x,\xi,r) \nabla\psi(x,\xi,r)\dx\dxi\dr = -\int_0^t\int_U \nabla u^{\left[\frac{m+1}{2}\right]}\psi(x,u(x,r),r)\dx\dr,
\end{equation}
which is a consequence of Definition~\ref{sol_def} and \cite[Lemma~3.6]{FehrmanGess}.

\section{Uniqueness of pathwise kinetic solutions}

The proof of uniqueness is based on the following formal calculation.  Let $\varphi\in\C^\infty(U)$ be the positive function defined by
\begin{equation}\label{u_aux} \left\{\begin{array}{ll} \Delta\varphi=-1 & \textrm{in}\;\;U, \\ \varphi=0 & \textrm{on}\;\;\partial U.\end{array}\right.\end{equation}
Let $u^1,u^2$ be pathwise kinetic solutions of \eqref{intro_eq}, and let $\chi^1$ and $\chi^2$ denote the corresponding kinetic functions.  Since properties of the kinetic function prove that
\begin{equation}\label{u_formal_0} \int_U\abs{u^1-u^2}\varphi v_{0,t} = \int_\mathbb{R}\int_U\left(\sgn(\xi)\chi^1+\sgn(\xi)\chi^2-2\chi^2\chi^2\right)\varphi v_{0,t},\end{equation}
it formally follows from the equation and the nonnegativity of the entropy defect measures that, for the parabolic defect measures $\{q^i\}_{i\in\{1,2\}}$,
\begin{equation}\label{u_formal}\begin{aligned} \partial_t\int_U\abs{u^1-u^2}v_{0,t} \leq & \int_\mathbb{R}\int_Um\abs{\xi}^{m-1}\abs{\chi^1-\chi^2}^2 \Delta (\varphi v_{0,t}) \\  & +4\int_\mathbb{R}\int_Um\abs{\xi}^{m-1}\nabla\chi^1\cdot \nabla\chi^2\varphi v_{0,t}\\ & -2\int_\mathbb{R}\int_U \left(q^1(x,\xi,u^2(x,t))+ q^2(x,\xi,u^1(x,t))\right)\varphi v_{0,t}.\end{aligned}\end{equation}
For sufficiently small times we have $\Delta(v_{0,t}\varphi)\leq 0$, which implies formally that the first term on the righthand side of \eqref{u_formal} is nonpositive.  The integration by parts formula \eqref{equation_ibp}, the definition of the parabolic defect measures, H\"older's inequality, and Young's inequality then imply formally that the sum of the second and third terms on the righthand side of \eqref{u_formal} is nonpositive.

We emphasize, however, that the rigorous justification of the cancellations observed in \eqref{u_formal} have remained an open question in the fast diffusion regime $m\in(0,1)$ due to the commutator errors described in \eqref{intro_mot_4} of the introduction.  Indeed, a regularization is necessary to justify the above computations, in which repeatedly the ill-defined product of Dirac delta distributions appears.  And, on the level of the convolution, the stochastic characteristics link the spatial and velocity variables, see \eqref{c_f}, and thereby introduce commutator errors that must be handled using new error estimates up to the boundary (cf.\ Propositions~\ref{aux_p} and \ref{aux_log} below).  Precisely, the optimal regularity estimates of \cite{FehrmanGess} degenerate near the boundary and, for this reason, we prove in Proposition~\ref{aux_log} below optimal estimates localized to the interior of the domain.  These estimates play an essential role throughout the proof of uniqueness, including to treat a  boundary layer that does not appear in the periodic framework of \cite{FehrmanGess}.

\begin{remark}  In the proof of Theorem~\ref{thm_unique} and for the remainder of the paper, after applying the integration by parts formula \eqref{equation_ibp}, we will frequently encounter derivatives of functions $f(x,\xi,r):U\times\mathbb{R}\times[0,\infty)\rightarrow\mathbb{R}$ evaluated at $\xi=u(x,r)$.  In order to simplify the notation, we make the convention that
$$\nabla_xf(x,u(x,r),r)=\left.\nabla_xf(x,\xi,r)\right|_{\xi=u(x,r)},$$
and analogous conventions for all possible derivatives.  That is, in every case, the notation indicates the derivative of $f$ evaluated at $(x,u(x,r),r)$ as opposed to the derivative of the full composition.  
\end{remark}

\begin{thm}\label{thm_unique}  Let $u_0^1,u_0^2\in L^2_+(U)$ be arbitrary, and let $u^1$ and $u^2$ be pathwise kinetic solutions of \eqref{intro_eq} in the sense of Definition~\ref{sol_def} with initial data $u_0^1$ and $u_0^2$ respectively.  Then, for $\varphi$ satisfying \eqref{u_aux}, for each $T>0$, there exists $C=C(m,U,T,z)>0$ such that
$$\norm{\left(u^1-u^2\right)\varphi}_{L^\infty\left([0,T],L^1(U)\right)}\leq C\norm{\left(u_0^1-u_0^2\right)\varphi}_{L^1(U)}.$$
\end{thm}

\begin{proof}  Let $u_0^1,u_0^2\in L^2_+(U)$, and suppose that $u^1$ and $u^2$ are pathwise kinetic solutions in the sense of Definition~\ref{sol_def} with initial data $u_0^1$ and $u_0^2$ respectively.  For each $i\in\{1,2\}$, we will write $\chi^i$ for the kinetic function and $(p^i,q^i)$ for the corresponding entropy and parabolic defect measures.

For each $\ve\in(0,1)$, let $\rho^\ve_d$ and $\rho^\ve_1$ be standard $d$-dimensional and $1$-dimensional convolution kernels on scale $\ve$.  For each $s\geq 0$, for the backward characteristics \eqref{c_char} and the weight defined in \eqref{c_w}, we define the smoothed transported kinetic function, for $(y,\eta,t)\in U_\ve\times\mathbb{R}\times[s,\infty)$,
\begin{equation}\begin{aligned}\label{tu_0} \tilde{\chi}_{s,t}^{i,\ve}(y,\eta):= & \int_\mathbb{R}\int_U\chi^i_t(x,\Xi^{x,\xi}_{s,t})\rho^\ve_d(x-y)\rho^\ve_1(\xi-\eta)\dx\dxi \\ = & \int_\mathbb{R}\int_{U}\chi^i_t(x,\xi)v_{s,t}(x)\rho^\ve_d(x-y)\rho^\ve_1\left(\Pi^{x,\xi}_{t,t-s}-\eta\right)\dx\dxi,\end{aligned}\end{equation}
where
\begin{equation}\label{tu_01}U_\ve:=\{\;x\in U\;|\;\dd(x,\partial U)>\ve\;\}.\end{equation}
In particular, for each $(y,\eta,t)\in U_\ve\times\mathbb{R}\times[s,\infty)$,
\begin{equation}\label{tu_0101}\lim_{\ve\rightarrow 0} \tilde{\chi}_{s,t}^{i,\ve}(y,\eta)=\tilde{\chi}^i_{s,t}(y,\eta)=\chi^i(y,\Xi^{y,\eta}_{s,t},t).\end{equation}
We observe that for each fixed $(y,\eta)\in U_\ve\times\mathbb{R}$ and $s\leq t\in(0,\infty)\setminus\mathcal{N}^i$, equation \eqref{transport_equation} can be applied to \eqref{tu_0}.  In what follows, to simplify the notation, we define, for each $\ve\in(0,1)$ and $s\geq 0$,
\begin{equation}\label{tu_00} \rho^\ve_{s,t}(x,y,\xi,\eta):=\rho^\ve_d(x-y)\rho^\ve_1\left(\Pi^{x,\xi}_{t,t-s}-\eta\right),\end{equation}
so that
$$\tilde{\chi}_{s,t}^{i,\ve}(y,\eta)= \int_\mathbb{R}\int_{U}\chi^i_t(x,\xi)v_{s,t}(x)\rho^\ve_{s,t}(x,y,\xi,\eta)\dx\dxi.$$

It is furthermore necessary to introduce a cutoff along the boundary of $U$, since the application of equation \eqref{transport_equation} to \eqref{tu_0} is only defined for points $y\in U$ of distance greater than $\ve\in(0,1)$ from the boundary.  For each $\beta\in(0,1)$, let $U_\beta$ denote the set from \eqref{tu_01}, and let $\textbf{1}_\beta$ be a smooth cutoff function of $U_\beta$ in $U$.  That is, using the smoothness of the domain, for each $\beta\in(0,1)$, fix a smooth function $\textbf{1}_\beta:\mathbb{R}^d\rightarrow[0,1]$ satisfying
\begin{equation}\label{tu_1}\textbf{1}_\beta(x)=\left\{\begin{array}{ll} 1 & \textrm{if}\;\;x\in U_\beta, \\ 0 & \textrm{if}\;\;x\in U\setminus U_\frac{\beta}{2},\end{array}\right.\end{equation}
and, for $C>0$ independent of $\beta\in(0,1)$, for each $x\in\mathbb{R}^d$,
\begin{equation}\label{tu_2}\abs{\nabla\textbf{1}_\beta(x)}+\beta\abs{\nabla^2\textbf{1}_\beta(x)}\leq \frac{C}{\beta}.\end{equation}
For each $\beta\in(0,1)$, for $\varphi$ satisfying \eqref{u_aux}, we define
\begin{equation}\label{tu_3}\varphi_\beta:=\textbf{1}_\beta\varphi.\end{equation}
It will be necessary to choose $\ve\in(0,1)$ sufficiently smaller than $\beta\in(0,1)$, so as to guarantee that the regularization \eqref{tu_0} is well-defined.

The proof will proceed in eight steps.  The first step introduces an approximation scheme that relies on a regularization in the spatial and velocity variables, as well as a time-splitting.  Step two analyzes the terms of \eqref{u_formal_0} involving the $\sgn$ function, and step three considers the mixed term.  We observe the cancellation due to the parabolic defect measures in step four.  In step five, we analyze error terms arising from the transport by the characteristics using commutator estimates and the time-splitting.  In step six, we pass to the limit with respect to the regularization in the spatial and velocity variables.  In step seven, we pass to the limit with respect to the time-splitting, in the sense that we consider a sequence of partitions whose mesh approaches zero.  Finally, in step eight, we pass to the limit $\beta\rightarrow 0$ and remove the boundary layer.

\textbf{Step 1:  The time-splitting and mollification.}  Henceforth, let $\beta\in(0,1)$ and $\ve\in(0,\frac{\beta}{2})$ be fixed but arbitrary.  For each $i\in\{1,2\}$, let $\mathcal{N}^i$ denote the zero set appearing in Definition~\ref{sol_def} and define $\mathcal{N}=\mathcal{N}^1\cup\mathcal{N}^2$.  Let $T\in[0,\infty)\setminus\mathcal{N}$ be arbitrary and let $\mathcal{P}:=\left\{\;0=t_0<t_1<\ldots<t_N=T\;\right\}$ be a fixed but arbitrary partition of $[0,T]$ with $\mathcal{P}\subset [0,T]\setminus\mathcal{N}$.  Properties of the kinetic function, the fact that the characteristics preserve the sign of the velocity variable, the change of variables factor \eqref{c_w}, and definition \eqref{tu_0} imply that
\begin{equation}\label{tu_4}\begin{aligned} & \left.\int_\mathbb{R}\int_U\abs{\chi^1_r(y,\eta)-\chi^2_r(y,\eta)}^2v_{0,r}\varphi_\beta\right|_{r=0}^T \\  & = \left.\int_\mathbb{R}\int_U\left(\chi^1_r\sgn(\eta)+\chi^2_r\sgn(\eta)-2\chi^1_r\chi^2_r\right)v_{0,r}\varphi_\beta\right|_{r=0}^T \\ & = \sum_{i=0}^{N-1}\left.\int_\mathbb{R}\int_U\left(\tilde{\chi}^1_{t_i,r}\sgn(\eta)+\tilde{\chi}^2_{t_i,r}\sgn(\eta)-2\tilde{\chi}^1_{t_i,r}\tilde{\chi}^2_{t_i,r}\right)v_{0,t_i}\varphi_\beta\right|_{r=t_i}^{t_{i+1}}\end{aligned}\end{equation}
and, for each $i\in\{0,\ldots,N-1\}$, it follows from the convergence \eqref{tu_0101} that
$$\begin{aligned} & \left.\int_\mathbb{R}\int_U\left(\tilde{\chi}^1_{t_i,r}\sgn(\eta)+\tilde{\chi}^2_{t_i,r}\sgn(\eta)-2\tilde{\chi}^1_{t_i,r}\tilde{\chi}^2_{t_i,r}\right)v_{0,t_i}\varphi_\beta\right|_{r=t_i}^{t_{i+1}} \\ &  =\lim_{\ve\rightarrow 0}\left.\int_\mathbb{R}\int_U(\tilde{\chi}^{1,\ve}_{t_i,r}\tilde{\sgn}^\ve_{t_i,r}(y,\eta)+\tilde{\chi}^{2,\ve}_{t_i,r}\tilde{\sgn}^\ve_{t_i,r}(y,\eta)-2\tilde{\chi}^{1,\ve}_{t_i,r}\tilde{\chi}^{2,\ve}_{t_i,r})v_{0,t_i}\varphi_\beta\right|_{r=t_i}^{t_{i+1}}\end{aligned}$$
where, for each $i\in\{0,\ldots,N-1\}$,
\begin{equation}\label{tu_5}\tilde{\sgn}^\ve_{t_i,r}(y,\eta):= \int_{\mathbb{R}}\int_U\sgn(\xi) v_{t_i,r}(x)\rho^\ve_d(x-y)\rho^\ve_1(\Pi^{x,\xi}_{r,r-t_i}-\eta)\dx\dxi.\end{equation}
The characteristic \eqref{c_char} and the weight \eqref{c_w} imply, for each $i\in\{1,\ldots,N-1\}$,
\begin{equation}\label{tu_6}\tilde{\sgn}^\ve_{t_i,r}(y,\eta)=\int_{\mathbb{R}}\int_U\sgn(\Xi^{x,\xi}_{r,r-t_i})\rho^\ve_d(x-y)\rho^\ve_1(\xi-\eta)=\int_{\mathbb{R}}\sgn(\xi)\rho^\ve_1(\xi-\eta).\end{equation}
In particular, the regularization \eqref{tu_5} is independent of $r\in[t_i,\infty)$ and $y\in U$ and satisfies
\begin{equation}\label{tu_010101}\lim_{\ve\rightarrow 0} \tilde{\sgn}^\ve_{t_i,r}(y,\eta)=\sgn(\eta).\end{equation}
It is nonetheless convenient to consider the regularization \eqref{tu_5}, since it clarifies an important cancellation property of the equation.

\textbf{Step 2:  The $\sgn$ terms.}  We will begin by analyzing the terms of \eqref{tu_4} involving the $\sgn$ function.  Henceforth, let $i\in\{0,\ldots,N-1\}$ be fixed but arbitrary.  We will write $(x,\xi)\in U\times\mathbb{R}$ for the integration variables defining $\tilde{\chi}^{1,\ve}_{t_i,r}$ and $(x',\xi')\in U\times\mathbb{R}$ for the integration variables defining $\tilde{\sgn}^\ve_{t_i,r}$.  We then define the corresponding convolution kernels and change of variables factors $\rho^{1,\ve}_{t_i,r}:=\rho^\ve_{t_i,r}(x,y,\xi,\eta)$ and $v^1_{t_i,r}:=v^1_{t_i,r}(x)$, and $\rho^{2,\ve}_{t_i,r}:=\rho^\ve_{t_i,r}(x',y,\xi',\eta)$ and $v^2_{t_i,r}:=v^1_{t_i,r}(x')$.

The definition of $\varphi_\beta$ in \eqref{tu_3}, the choice of $\ve\in(0,\frac{\beta}{2})$, and \eqref{tu_6} imply that, using the test function $\rho(x,\xi)=\rho^\ve_d(x-y)\rho^\ve_1(\xi-\eta)$ in \eqref{transport_equation},
\begin{equation}\begin{aligned}\label{tu_7}  \left.\int_\mathbb{R}\int_U \tilde{\chi}^{1,\ve}_{t_i,r}\tilde{\sgn}^\ve_{t_i,r}v_{0,t_i} \varphi_\beta\right|_{r=t_i}^{t_{i+1}} & = \int_{t_i}^{t_{i+1}}\int_\mathbb{R}\int_U\left(\int_{\mathbb{R}}\int_{U}m\abs{\xi}^{m-1}\chi^1_r \Delta_x \left(\rho^{1,\ve}_{t_i,r}v^1_{t_i,r}\right)\right) \tilde{\sgn}^\ve_{t_i,r}v_{0,t_i}\varphi_\beta  \\ & \quad -\int_{t_i}^{t_{i+1}}\int_\mathbb{R}\int_U\left(\int_\mathbb{R}\int_U(p^1_r+q^1_r)v^1_{t_i,r}\partial_\xi\left(\rho^{1,\ve}_{t_i,r}\right)\right)\tilde{\sgn}^{\ve}_{t_i,r}v_{0,t_i}\varphi_\beta.\end{aligned}\end{equation}
For the first term of \eqref{tu_7}, the identity
$$\begin{aligned}  \nabla_x\rho^\ve_{t_i,r}(x,y,\xi,\eta) & = -\nabla_y\rho^\ve_d(x-y)\rho^\ve_1(\Pi^{x,\xi}_{r,r-t_i}-\eta)-\rho^\ve_d(x-y)\partial_\eta\rho^\ve_1(\Pi^{x,\xi}_{r,r-t_i}-\eta)\nabla_x\Pi^{x,\xi}_{r,r-t_i} \\ & =-\nabla_y\rho^\ve_{t_i,r}(x,y,\xi,\eta)-\partial_\eta\rho^\ve_{t_i,r}(x,y,\xi,\eta)\nabla_x\Pi^{x,\xi}_{r,r-t_i},\end{aligned}$$
implies that
$$\begin{aligned}& \int_{t_i}^{t_{i+1}}\int_\mathbb{R}\int_U\left(\int_{\mathbb{R}}\int_{U}m\abs{\xi}^{m-1}\chi^1_r \Delta_x \left(\rho^{1,\ve}_{t_i,r}v^1_{t_i,r}\right)\right) \tilde{\sgn}^\ve_{t_i,r}v_{0,t_i}\varphi_\beta \\ &   = \int_{t_i}^{t_{i+1}}\int_{\mathbb{R}^2}\int_{U^2}m\abs{\xi}^{m-1}\chi^1_r \nabla_x\left(\rho^{1,\ve}_{t_i,r}v^1_{t_i,r}\right)\cdot \nabla_y\left(\tilde{\sgn}^\ve_{t_i,r}v_{0,t_i}\varphi_\beta\right) \\ &  \quad + \int_{t_i}^{t_{i+1}}\int_{\mathbb{R}^2}\int_{U^2}m\abs{\xi}^{m-1}\chi^1_r \nabla_x\cdot\left(\rho^{1,\ve}_{t_i,r}v^1_{t_i,r}\nabla_x\Pi^{x,\xi}_{r,r-t_i}\right) \partial_\eta\left(\tilde{\sgn}^\ve_{t_i,r}v_{0,t_i}\varphi_\beta\right) \\ & \quad + \int_{t_i}^{t_{i+1}}\int_{\mathbb{R}^2}\int_{U^2}m\abs{\xi}^{m-1}\chi^1_r \nabla_x \cdot \left(\rho^{1,\ve}_{t_i,r}\nabla_xv^1_{t_i,r}\right) \tilde{\sgn}^\ve_{t_i,r}v_{0,t_i}\varphi_\beta. \end{aligned}$$

After adding and subtracting the gradients $\nabla_{x'}\Pi^{x',\xi'}_{r,r-t_i}$ and $\nabla_{x'}v^2_{t_i,r}(x')$,
\begin{equation}\label{tu_90}\begin{aligned} & \int_{t_i}^{t_{i+1}}\int_\mathbb{R}\int_U\left(\int_{\mathbb{R}}\int_{U}m\abs{\xi}^{m-1}\chi^1_r \Delta_x \left(\rho^{1,\ve}_{t_i,r}v^1_{t_i,r}\right)\right) \tilde{\sgn}^\ve_{t_i,r}v_{0,t_i}\varphi_\beta \\ &   =\int_{t_i}^{t_{i+1}}\int_\mathbb{R}\int_U\left(\int_{\mathbb{R}}\int_{U}m\abs{\xi}^{m-1}\chi^1_r \nabla_x\left(\rho^{1,\ve}_{t_i,r}v^1_{t_i,r}\right)\right)\cdot \tilde{\sgn}^\ve_{t_i,r}\nabla_y\left(v_{0,t_i}\varphi_\beta\right) \\  &  \quad - \int_{t_i}^{t_{i+1}}\int_{\mathbb{R}^3}\int_{U^3} m\abs{\xi}^{m-1}\chi^1_r\nabla_x\left(\rho^{1,\ve}_{t_i,r}v^1_{t_i,r}\right)\cdot\sgn(\xi')\nabla_{x'}\left(\rho^{2,\ve}_{t_i,r}v^2_{t_i,r}\right)v_{0,t_i}\varphi_\beta + \textrm{Err}^{1,1}_{i,\ve},\end{aligned}\end{equation}
for the error term
\begin{equation}\label{tu_10} \textrm{Err}^{1,1}_{i,\ve}: = \int_{t_i}^{t_{i+1}}\int_{\mathbb{R}^3}\int_{U^3}\left(\textrm{err}^{1,1}_{i,\ve}\right)v_{0,t_i}\varphi_\beta,\end{equation}
where
$$\begin{aligned}  \textrm{err}^{1,1}_{i,\ve}& :=m\abs{\xi}^{m-1}\chi^1_r \nabla_x\cdot\left(\rho^{1,\ve}_{t_i,r}\left(\nabla_xv^1_{t_i,r}-\nabla_{x'}v^2_{t_i,r}\right)\right)\sgn(\xi')\rho^{2,\ve}_{t_i,r}v^2_{t_i,r} \\ & \quad +m\abs{\xi}^{m-1}\chi^1_r \nabla_x\cdot \left(\rho^{1,\ve}_{t_i,r}v^1_{t_i,r}\left(\nabla_x\Pi^{x,\xi}_{r,r-t_i}-\nabla_{x'}\Pi^{x',\xi'}_{r,r-t_i}\right)\right)\sgn(\xi')\partial_\eta\rho^{2,\ve}_{t_i,r}v^2_{t_i,r}.\end{aligned}$$
Furthermore, in view of \eqref{tu_6}, the second term on the righthand side of \eqref{tu_90} vanishes after integrating by parts in $x'\in U$.  Therefore, from \eqref{tu_90} and \eqref{tu_10},
\begin{equation}\begin{aligned}\label{tu_11} & \int_{t_i}^{t_{i+1}}\int_\mathbb{R}\int_U\left(\int_{\mathbb{R}}\int_{U}m\abs{\xi}^{m-1}\chi^1_r \Delta_x \left(\rho^{1,\ve}_{t_i,r}v^1_{t_i,r}\right)\right) \tilde{\sgn}^\ve_{t_i,r}v_{0,t_i}\varphi_\beta \\ &   = \int_{t_i}^{t_{i+1}}\int_\mathbb{R}\int_U\left(\int_{\mathbb{R}}\int_{U}m\abs{\xi}^{m-1}\chi^1_r \nabla_x\left(\rho^{1,\ve}_{t_i,r}v^1_{t_i,r}\right)\right)\cdot \tilde{\sgn}^\ve_{t_i,r}\nabla_y\left(v_{0,t_i}\varphi_\beta\right) + \textrm{Err}^{1,1}_{i,\ve}.\end{aligned}\end{equation}

For the second term of \eqref{tu_7}, we will use the identity
$$\partial_\xi\rho^\ve_{t_i,r}(x,y,\xi,\eta)=-\partial_\eta\rho^\ve_{t_i,r}(x,y,\xi,\eta)\partial_\xi\Pi^{x,\xi}_{r,r-t_i},$$
which implies that, after adding and subtracting the derivative $\partial_{\xi'}\Pi^{x',\xi'}_{r,r-t_i}$,
\begin{equation}\begin{aligned}\label{tu_13} & \int_{t_i}^{t_{i+1}}\int_\mathbb{R}\int_U\left(\int_\mathbb{R}\int_U(p^1_r+q^1_r)v^1_{t_i,r}\partial_\xi\left(\rho^{1,\ve}_{t_i,r}\right)\right)\tilde{\sgn}^\ve_{t_i,r}v_{0,t_i}\varphi_\beta  \\ & = -\int_{t_i}^{t_{i+1}}\int_\mathbb{R}\int_U\left(\int_{\mathbb{R}}\int_{U^2}(p^1_r+q^1_r)\rho^{1,\ve}_{t_i,r}v^1_{t_i,r}\sgn(\xi')\partial_{\xi'}\left(\rho^{2,\ve}_{t_i,r}\right)v^2_{t_i,r}\right)v_{0,t_i}\varphi_\beta + \textrm{Err}^{1,2}_{i,\ve},\end{aligned}\end{equation}
for the error term
\begin{equation}\label{tu_14}\textrm{Err}^{1,2}_{i,\ve}:=\int_{t_i}^{t_{i+1}}\int_{\mathbb{R}^3}\int_{U^3}\left(\textrm{err}^{1,2}_{i,\ve}\right)v_{0,t_i}\varphi_\beta,\end{equation}
where
$$ \textrm{err}^{1,2}_{i,\ve} :=(p^1_r+q^1_r)\left(\rho^{1,\ve}_{t_i,r}v^1_{t_i,r}\left(\partial_\xi\Pi^{x,\xi}_{r,r-t_i}-\partial_{\xi'}\Pi^{x',\xi'}_{r,r-t_i}\right)\right)\sgn(\xi')\partial_\eta\rho^{2,\ve}_{t_i,r}v^2_{t_i,r}. $$
Using the distributional inequality $\partial_{\xi'}\sgn(\xi')=2\delta_0(\xi')$, the first term of \eqref{tu_13} satisfies, after integrating by parts,
\begin{equation}\label{tu_133} 2\int_{t_i}^{t_{i+1}}\int_{\mathbb{R}^2}\int_{U^3}(p^1_r+q^1_r)\rho^{1,\ve}_{t_i,r}(x,y,\xi,\eta)v^1_{t_i,r}\rho^{2,\ve}_{t_i,r}(x',y,0,\eta)v^2_{t_i,r}v_{0,t_i}\varphi_\beta.\end{equation}
Therefore, in view of \eqref{tu_7}, \eqref{tu_11}, \eqref{tu_13}, and \eqref{tu_133},
\begin{equation}\begin{aligned}\label{tu_15} & \left.\int_\mathbb{R}\int_U \tilde{\chi}^{1,\ve}_{t_i,r}\tilde{\sgn}^\ve_{t_i,r}(y,\eta)v_{0,t_i}\varphi_\beta\right|_{r=t_i}^{t_{i+1}} \\ & = \int_{t_i}^{t_{i+1}}\int_\mathbb{R}\int_U\left(\int_{\mathbb{R}}\int_{U}m\abs{\xi}^{m-1}\chi^1_r \nabla_x\left(\rho^{1,\ve}_{t_i,r}v^1_{t_i,r}\right)\right)\cdot \tilde{\sgn}^\ve_{t_i,r}\nabla_y\left(v_{0,t_i}\varphi_\beta\right) \\ & \quad - 2\int_{t_i}^{t_{i+1}}\int_\mathbb{R}\int_U\left(\int_{\mathbb{R}}\int_{U^2}(p^1_r+q^1_r)\rho^{1,\ve}_{t_i,r}v^1_{t_i,r}\rho^{2,\ve}_{t_i,r}(x',y,0,\eta)v^2_{t_i,r}\right)v_{0,t_i}\varphi_\beta +\textrm{Err}^{1,1}_{i,\ve}-\textrm{Err}^{1,2}_{i,\ve}.\end{aligned}\end{equation}
Similarly, now writing $(x',\xi')\in U\times\mathbb{R}$ for the integration variables defining $\tilde{\chi}^{2,\ve}_{t_i,r}$, and defining $\textrm{Err}^{2,1}_{i,\ve}$ and $\textrm{Err}^{2,2}_{i,\ve}$ in exact analogy with \eqref{tu_10} and \eqref{tu_14},
\begin{equation}\begin{aligned}\label{tu_16} & \left.\int_\mathbb{R}\int_U \tilde{\chi}^{2,\ve}_{t_i,r}\tilde{\sgn}^\ve_{t_i,r}(y,\eta)v_{0,t_i}\varphi_\beta\right|_{r=t_i}^{t_{i+1}} \\ & =\int_{t_i}^{t_{i+1}}\int_\mathbb{R}\int_U\left(\int_{\mathbb{R}}\int_{U}m\abs{\xi'}^{m-1}\chi^2_r \nabla_{x'}\left(\rho^{2,\ve}_{t_i,r}v^2_{t_i,r}\right)\right)\cdot \tilde{\sgn}^\ve_{t_i,r}\nabla_y\left(v_{0,t_i}\varphi_\beta\right) \\ & \quad - 2\int_{t_i}^{t_{i+1}}\int_\mathbb{R}\int_U\left(\int_{\mathbb{R}}\int_{U^2}(p^2_r+q^2_r)\rho^{2,\ve}_{t_i,r}v^2_{t_i,r}\rho^{1,\ve}_{t_i,r}(x,y,0,\eta)v^1_{t_i,r}\right)v_{0,t_i}\varphi_\beta +\textrm{Err}^{2,1}_{i,\ve}-\textrm{Err}^{2,2}_{i,\ve}.\end{aligned}\end{equation}
This completes our initial analysis of the $\sgn$ terms.

\textbf{Step 3:  The mixed term.}  We will now analyze the mixed term of \eqref{tu_4}.  We will write $(x,\xi)\in U\times\mathbb{R}$ for the integration variables defining $\tilde{\chi}^{1,\ve}_{t_i,r}$, and we will write $(x',\xi')\in U\times\mathbb{R}$ for the integration variables defining $\tilde{\chi}^{2,\ve}_{t_i,r}$.  We make the same conventions concerning the convolution kernels $\{\rho^{j,\ve}_{t_i,r}\}_{j\in\{1,2\}}$ and weights $\{v^j_{t_i,r}\}_{j\in\{1,2\}}$.

The definition of $\varphi_\beta$ in \eqref{tu_3}, the choice of $\ve\in(0,\frac{\beta}{2})$, and \eqref{tu_6} imply that $\rho(x,\xi)=\rho^\ve_d(x-y)\rho^\ve_1(\xi-\eta)$ is an admissible test function for \eqref{transport_equation} for each $(y,\eta)\in U\times\mathbb{R}$ in the support of $\varphi_\beta$.  Therefore, a virtually identical analysis as that leading from \eqref{tu_7} to \eqref{tu_15} and \eqref{tu_16} proves that
\begin{equation}\label{tu_18}\begin{aligned}   \left.\int_\mathbb{R}\int_U \tilde{\chi}^{1,\ve}_{t_i,r}\tilde{\chi}^{2,\ve}_{t_i,r}v_{0,t_i}\varphi_\beta\right|_{r=t_i}^{t_{i+1}} & = \int_{t_i}^{t_{i+1}}\int_\mathbb{R}\int_U\left(\int_{\mathbb{R}}\int_{U}m\abs{\xi}^{m-1}\chi^1_r \nabla_x\left(\rho^{1,\ve}_{t_i,r}v^1_{t_i,r}\right)\right)\cdot \tilde{\chi}^{2,\ve}_{t_i,r}\nabla_y\left(v_{0,t_i}\varphi_\beta\right) \\ &\quad + \int_{t_i}^{t_{i+1}}\int_\mathbb{R}\int_U\left(\int_{\mathbb{R}}\int_{U}m\abs{\xi'}^{m-1}\chi^2_r \nabla_{x'}\left(\rho^{2,\ve}_{t_i,r}v^2_{t_i,r}\right)\right)\cdot \tilde{\chi}^{1,\ve}_{t_i,r}\nabla_y\left(v_{0,t_i}\varphi_\beta\right) \\ &  \quad -\int_{t_i}^{t_{i+1}}\int_{\mathbb{R}^3}\int_{U^3} m\abs{\xi}^{m-1}\chi^1_r\nabla_x\left(\rho^{1,\ve}_{t_i,r}v^1_{t_i,r}\right)\cdot \chi^2_r \nabla_{x'}\left(\rho^{2,\ve}_{t_i,r}v^2_{t_i,r}\right)v_{0,t_i}\varphi_\beta \\ & \quad -\int_{t_i}^{t_{i+1}}\int_{\mathbb{R}^3}\int_{U^3} m\abs{\xi'}^{m-1}\chi^1_r\nabla_x\left(\rho^{1,\ve}_{t_i,r}v^1_{t_i,r}\right)\cdot \chi^2_r \nabla_{x'}\left(\rho^{2,\ve}_{t_i,r}v^2_{t_i,r}\right)v_{0,t_i}\varphi_\beta \\ &  \quad -\int_{t_i}^{t_{i+1}}\int_{\mathbb{R}^3}\int_{U^3}\left(\left(p^1_r+q^1_r\right)\rho^{1,\ve}_{t_i,r}v^1_{t_i,r}\cdot \left(\partial_{\xi'}\chi^2_r\right)\rho^{2,\ve}_{t_i,r}v^2_{t_i,r}\right)v_{0,t_i}\varphi_\beta   \\ &  \quad -\int_{t_i}^{t_{i+1}}\int_{\mathbb{R}^3}\int_{U^3}\left(\left(p^2_r+q^2_r\right)\rho^{2,\ve}_{t_i,r}v^2_{t_i,r}\cdot \left(\partial_\xi\chi^1_r\right)\rho^{1,\ve}_{t_i,r}v^1_{t_i,r}\right)v_{0,t_i}\varphi_\beta \\ & \quad + \textrm{Err}^{1,3}_{i,\ve}-\textrm{Err}^{1,4}_{i,\ve}+\textrm{Err}^{2,3}_{i,\ve}-\textrm{Err}^{2,4}_{i,\ve},\end{aligned}\end{equation}
for the error terms
\begin{equation}\label{tu_19} \textrm{Err}^{1,3}_{i,\ve} := \int_{t_i}^{t_{i+1}}\int_{\mathbb{R}^3}\int_{U^3}\left(\textrm{err}^{1,3}_{i,\ve}\right)v_{0,t_i}\varphi_\beta,\end{equation}
where
$$\begin{aligned} & \textrm{err}^{1,3}_{i,\ve} :=m\abs{\xi}^{m-1}\chi^1_r \nabla_x\cdot\left(\rho^{1,\ve}_{t_i,r}\left(\nabla_xv^1_{t_i,r}-\nabla_{x'}v^2_{t_i,r}\right)\right)\chi^2_r\rho^{2,\ve}_{t_i,r}v^2_{t_i,r} \\ &  \quad +m\abs{\xi}^{m-1}\chi^1_r \nabla_x\cdot\left(\rho^{1,\ve}_{t_i,r}v^1_{t_i,r}\left(\nabla_x\Pi^{x,\xi}_{r,r-t_i}-\nabla_{x'}\Pi^{x',\xi'}_{r,r-t_i}\right)\right)\chi^2_r\partial_\eta\rho^{2,\ve}_{t_i,r}v^2_{t_i,r},  \end{aligned}$$
and for
\begin{equation}\label{tu_20} \textrm{Err}^{1,4}_{i,\ve} := \int_{t_i}^{t_{i+1}}\int_{\mathbb{R}^3}\int_{U^3}\left(\textrm{err}^{1,4}_{i,\ve}\right)v_{0,t_i}\varphi_\beta,\end{equation}
where
$$\textrm{err}^{1,4}_{i,\ve}:=\left(p^1_r+q^1_r\right)\rho^{1,\ve}_{t_i,r}v^1_{t_i,r}\left(\partial_\xi\Pi^{x,\xi}_{r,r-t_i}-\partial_{\xi'}\Pi^{x',\xi'}_{r,r-t_i}\right)\chi^2_r\partial_\eta\rho^{2,\ve}_{t_i,r}v^2_{t_i,r},$$
and with $\textrm{Err}^{2,3}_{i,\ve}$ and $\textrm{Err}^{2,4}_{i,\ve}$ defined analogously to \eqref{tu_19} and \eqref{tu_20}.  This concludes the initial analysis of the mixed term.

\textbf{Step 4:  The cancellation of the parabolic defect measures.}  We will now observe the cancellation from the parabolic defect measures.  It follows from \eqref{tu_15}, \eqref{tu_16}, \eqref{tu_18}, and the distributional equality $\partial_\xi \chi^j_t(x,\xi)=\delta_0(\xi)-\delta_0(\xi-u^j(x,t))$ that
\begin{equation}\label{tu_21}\begin{aligned}
& \left.\int_\mathbb{R}\int_U\left(\tilde{\chi}^{1,\ve}_{t_i,r}\tilde{\sgn}^\ve_{t_i,r}(y,\eta)+\tilde{\chi}^{2,\ve}_{t_i,r}\tilde{\sgn}^\ve_{t_i,r}(y,\eta)-2\tilde{\chi}^{1,\ve}_{t_i,r}\tilde{\chi}^{2,\ve}_{t_i,r}\right)v_{0,t_i}\varphi_\beta\right|_{r=t_i}^{t_{i+1}}
 \\ & =\int_{t_i}^{t_{i+1}}\int_{\mathbb{R}^2}\int_{U^2}m\abs{\xi}^{m-1}\chi^1_r \nabla_x\left(\rho^{1,\ve}_{t_i,r}v^1_{t_i,r}\right)\cdot \left(\tilde{\sgn}^\ve_{t_i,r}-2\tilde{\chi}^{2,\ve}_{t_i,r}\right)\nabla_y\left(v_{0,t_i}\varphi_\beta\right)
 \\ &  \quad + \int_{t_i}^{t_{i+1}}\int_{\mathbb{R}^2}\int_{U^2}m\abs{\xi'}^{m-1}\chi^2_r \nabla_{x'}\left(\rho^{2,\ve}_{t_i,r}v^2_{t_i,r}\right)\cdot \left(\tilde{\sgn}^\ve_{t_i,r}-2\tilde{\chi}^{1,\ve}_{t_i,r}\right)\nabla_y\left(v_{0,t_i}\varphi_\beta\right)
 \\ &\quad + 2\int_{t_i}^{t_{i+1}}\int_{\mathbb{R}^3}\int_{U^3} m\abs{\xi}^{m-1}\chi^1_r\nabla_x\left(\rho^{1,\ve}_{t_i,r}v^1_{t_i,r}\right)\cdot \chi^2_r \nabla_{x'}\left(\rho^{2,\ve}_{t_i,r}v^2_{t_i,r}\right)v_{0,t_i}\varphi_\beta
 \\ &\quad + 2\int_{t_i}^{t_{i+1}}\int_{\mathbb{R}^3}\int_{U^3} m\abs{\xi'}^{m-1}\chi^1_r\nabla_x\left(\rho^{1,\ve}_{t_i,r}v^1_{t_i,r}\right)\cdot \chi^2_r \nabla_{x'}\left(\rho^{2,\ve}_{t_i,r}v^2_{t_i,r}\right)v_{0,t_i}\varphi_\beta
 \\ & \quad - 2\int_{t_i}^{t_{i+1}}\int_{\mathbb{R}^2}\int_{U^3}\left(\left(p^1_r+q^1_r\right)\rho^{1,\ve}_{t_i,r}v^1_{t_i,r} \rho^{2,\ve}_{t_i,r}(x',y,u^2(x',t),\eta)v^2_{t_i,r}\right)v_{0,t_i}\varphi_\beta
 \\ & \quad - 2 \int_{t_i}^{t_{i+1}}\int_{\mathbb{R}^2}\int_{U^3}\left(\left(p^2_r+q^2_r\right)\rho^{2,\ve}_{t_i,r}v^2_{t_i,r}\cdot \rho^{1,\ve}_{t_i,r}(x,y,u^1(x,t),\eta)v^1_{t_i,r}\right)v_{0,t_i}\varphi_\beta
 \\ & \quad + \sum_{j=1}^2 \textrm{Err}^{j,1}_{i,\ve}-\textrm{Err}^{j,2}_{i,\ve}-2\textrm{Err}^{j,3}_{i,\ve}+2\textrm{Err}^{j,4}_{i,\ve}.\end{aligned}\end{equation}
The integration by parts formula \eqref{equation_ibp}, Proposition~\ref{aux_p} below, and an approximation argument imply that, after applying H\"older's inequality and Young's inequality,
$$\begin{aligned} & 4\int_{t_i}^{t_{i+1}}\int_{\mathbb{R}^3}\int_{U^3} m\abs{\xi}^{\frac{m-1}{2}}\abs{\xi'}^{\frac{m-1}{2}}\chi^1_r\nabla_x\left(\rho^{1,\ve}_{t_i,r}v^1_{t_i,r}\right)\cdot \chi^2_r \nabla_{x'}\left(\rho^{2,\ve}_{t_i,r}v^2_{t_i,r}\right)v_{0,t_i}\varphi_\beta \\ & = \frac{4m}{(m+1)^2}\int_{t_i}^{t_{i+1}}\int_\mathbb{R}\int_{U^3}\nabla\left(u^1\right)^{\left[\frac{m+1}{2}\right]}\cdot\nabla\left(u^2\right)^{\left[\frac{m+1}{2}\right]}\overline{\rho}^{1,\ve}_{t_i,r}v^1_{t_i,r}\overline{\rho}^{2,\ve}_{t_i,r}v^2_{t_i,r},\end{aligned}$$
where, for each $j\in\{1,2\}$, for $\rho^\ve_{t_i,r}$ from \eqref{tu_00},
$$\overline{\rho}^{j,\ve}_{t_i,r}(x,\xi,\eta):=\rho^\ve_{t_i,r}(x,u^j(x,r),y,\eta).$$
Therefore, the definition and nonnegativity of the entropy and parabolic defect measures imply with H\"older's inequality and Young's inequality that
\begin{equation}\label{tu_022}\begin{aligned} & 4\int_{t_i}^{t_{i+1}}\int_{\mathbb{R}^3}\int_{U^3}m\abs{\xi}^{\frac{m-1}{2}}\abs{\xi'}^{\frac{m-1}{2}}\chi^1_r\nabla_x\left(\rho^{1,\ve}_{t_i,r}v^1_{t_i,r}\right)\cdot \chi^2_r \nabla_{x'}\left(\rho^{2,\ve}_{t_i,r}v^2_{t_i,r}\right)v_{0,t_i}\varphi_\beta \\ & \leq 2\int_{t_i}^{t_{i+1}}\int_{\mathbb{R}^2}\int_{U^3}\left(p^1_r+q^1_r\right)\rho^{1,\ve}_{t_i,r}v^1_{t_i,r}\cdot \overline{\rho}^{2,\ve}_{t_i,r}v^2_{t_i,r}v_{0,t_i}\varphi_\beta \\ & \quad + 2 \int_{t_i}^{t_{i+1}}\int_{\mathbb{R}^2}\int_{U^3}\left(p^2_r+q^2_r\right)\rho^{2,\ve}_{t_i,r}v^2_{t_i,r}\cdot \overline{\rho}^{1,\ve}_{t_i,r}v^1_{t_i,r}v_{0,t_i}\varphi_\beta.\end{aligned}\end{equation}
It follows from the identity
$$m\left(\abs{\xi}^{m-1}+\abs{\xi'}^{m-1}-2\abs{\xi}^{\frac{m-1}{2}}\abs{\xi'}^{\frac{m-1}{2}}\right)=m\left(\abs{\xi}^{\frac{m-1}{2}}-\abs{\xi'}^{\frac{m-1}{2}}\right)^2,$$
that, after adding and subtracting $\abs{\xi}^{\frac{m-1}{2}}\abs{\xi'}^{\frac{m-1}{2}}$ in the third and fourth terms of \eqref{tu_21} and applying \eqref{tu_022},
\begin{equation}\label{tu_24}\begin{aligned}& \left.\int_\mathbb{R}\int_U\left(\tilde{\chi}^{1,\ve}_{t_i,r}\tilde{\sgn}^\ve_{t_i,r}(y,\eta)+\tilde{\chi}^{2,\ve}_{t_i,r}\tilde{\sgn}^\ve_{t_i,r}(y,\eta)-2\tilde{\chi}^{1,\ve}_{t_i,r}\tilde{\chi}^{2,\ve}_{t_i,r}\right)v_{0,t_i}\varphi_\beta\right|_{r=t_i}^{t_{i+1}}
\\ & \leq \int_{t_i}^{t_{i+1}}\int_{\mathbb{R}^2}\int_{U^2}m\abs{\xi}^{m-1}\chi^1_r \nabla_x\left(\rho^{1,\ve}_{t_i,r}v^1_{t_i,r}\right)\cdot \left(\tilde{\sgn}^\ve_{t_i,r}-2\tilde{\chi}^{2,\ve}_{t_i,r}\right)\nabla_y\left(v_{0,t_i}\varphi_\beta\right)
\\ & \quad + \int_{t_i}^{t_{i+1}}\int_{\mathbb{R}^2}\int_{U^2}m\abs{\xi'}^{m-1}\chi^2_r \nabla_{x'}\left(\rho^{2,\ve}_{t_i,r}v^2_{t_i,r}\right)\cdot \left(\tilde{\sgn}^\ve_{t_i,r}-2\tilde{\chi}^{1,\ve}_{t_i,r}\right)\nabla_y\left(v_{0,t_i}\varphi_\beta\right)
\\ & \quad + \sum_{j=1}^2 \left(\textrm{Err}^{j,1}_{i,\ve}-\textrm{Err}^{j,2}_{i,\ve}-2\textrm{Err}^{j,3}_{i,\ve}+2\textrm{Err}^{j,4}_{i,\ve}\right)+\textrm{Err}^5_{i,\ve}, \end{aligned}\end{equation}
for the error term
\begin{equation}\label{tu_23} \textrm{Err}^5_{i,\ve}:=2\int_{t_i}^{t_{i+1}}\int_{\mathbb{R}^3}\int_{U^3}\left(\textrm{err}^5_{i,\ve}\right)v_{0,t_i}\varphi_\beta,\end{equation}
where
\begin{equation}\label{error_5}\textrm{err}^5_{i,\ve}:=m\left(\abs{\xi}^{\frac{m-1}{2}}-\abs{\xi'}^{\frac{m-1}{2}}\right)^2\chi^1_r\nabla_x\left(\rho^{1,\ve}_{t_i,r}v^1_{t_i,r}\right)\cdot \chi^2_r \nabla_{x'}\left(\rho^{2,\ve}_{t_i,r}v^2_{t_i,r}\right).\end{equation}
We will now analyze the error terms.

\textbf{Step 5:  The error terms.}  The first four error terms will be controlled using the continuity of the noise.  Precisely, for each $\delta\in(0,1)$, define
$$\omega(\delta;T):=\sup_{\stackrel{0\leq s\leq t\leq T}{\abs{s-t}\leq \delta}}\abs{z_t-z_s},$$
where the continuity of the noise implies that $\lim_{\delta\rightarrow 0}\omega(\delta;T)=0$.
Observe from the definition the weight \eqref{c_w} that, using the regularity of the coefficents, for $C=C(T)>0$, for each $x,x'\in U$,
\begin{equation}\label{tu_25} \max_{r\in[t_i,t_{i+1}]}\abs{\nabla_x v_{t_i,r}(x)-\nabla_{x'} v_{t_i,r}(x')}\leq  C\omega(\abs{t_{i+1}-t_i};T)\abs{x-x'},\end{equation}
and, using the definition of the characteristic \eqref{c_char}, for each $(x,x',\xi,\xi')\in U^2\times\mathbb{R}^2$,
\begin{equation}\label{tu_26} \max_{r\in[t_i,t_{i+1}]}\abs{\partial_\xi\Pi^{x,\xi}_{r,r-t_i}-\partial_{\xi'}\Pi^{x',\xi'}_{r,r-t_i}}\leq C\omega(\abs{t_{i+1}-t_i};T)\abs{x-x'},\end{equation}
and
\begin{equation} \label{tu_27} \max_{r\in[t_i,t_{i+1}]}\abs{\nabla_x\Pi^{x,\xi}_{r,r-t_i}-\nabla_{x'}\Pi^{x',\xi'}_{r,r-t_i}} \leq C\omega(\abs{t_{i+1}-t_i};T)\left(\abs{\xi-\xi'}+\min\{\abs{\xi},\abs{\xi'}\}\abs{x-x'}\right).\end{equation}
Observe that the definition of the convolution kernel \eqref{tu_00} implies that, whenever
$$\rho^\ve_{s,t}(x,y,\xi,\eta)\rho^\ve_{s,t}(x',y,\xi',\eta)\neq 0,$$
we have, for $C=C(T)>0$, 
\begin{equation}\label{tu_000} \abs{x-x'}\leq C\ve\;\;\textrm{and}\;\; \abs{\xi-\xi'}\leq C\left(1+\min\{\abs{\xi},\abs{\xi'}\}\right)\ve.\end{equation}
Finally, it follows from the definition of the convolution kernel \eqref{tu_00} and the change of variables factor \eqref{c_w} that there exists $C=C(T)>0$ such that, for each $(x,\xi,r)\in U\times\mathbb{R}\times[t_i,t_{i+1}]$,
\begin{equation}\label{u_01000}\int_\mathbb{R}\int_{\mathbb{R}^d}\abs{\partial_\eta\rho^\ve_{t_i,r}(x,y,\xi\,\eta)}\dy\deta\leq \frac{C}{\ve}.\end{equation}

\emph{$\{\textrm{Err}^{j,1}_{i,\ve}\}_{j\in\{1,2\}}$ and $\{\textrm{Err}^{j,3}_{i,\ve}\}_{j\in\{1,2\}}$}:  We first treat the errors $\{\textrm{Err}^{j,1}_{i,\ve}\}_{j\in\{1,2\}}$ from \eqref{tu_10} and $\{\textrm{Err}^{j,3}_{i,\ve}\}_{j\in\{1,2\}}$ from \eqref{tu_19}.  For the case of  $\textrm{Err}^{1,1}_{i,\ve}$, first applying the integration by parts formula \eqref{equation_ibp}, which is justified using Lemma~\ref{lem_bdry} below, taking the absolute value, and using the boundedness of $\sgn$, $v_{t_i,r}$, $v_{0,t_i}$, and $\varphi_\beta$, for $C=C(m,U,T)>0$,
$$\begin{aligned} \abs{\textrm{Err}^{1,1}_{i,\ve}}& \leq C\int_{t_i}^{t_{i_1}}\int_{\mathbb{R}^2}\int_{U^3}\abs{u^1}^{\frac{m-1}{2}}\abs{\nabla\left(u^1\right)^{\left[\frac{m+1}{2}\right]}}\overline{\rho}^{1,\ve}_{t_i,r}\abs{\nabla_x\Pi^{x,\xi}_{r,r-t_i}-\nabla_{x'}\Pi^{x',\xi'}_{r,r-t_i}}\abs{\partial_\eta\rho^{2,\ve}_{t_i,r}} \\ & \quad + C\int_{t_i}^{t_{i_1}}\int_{\mathbb{R}^2}\int_{U^3}\abs{u^1}^{\frac{m-1}{2}}\abs{\nabla\left(u^1\right)^{\left[\frac{m+1}{2}\right]}}\overline{\rho}^{1,\ve}_{t_i,r}\abs{\nabla_xv^1_{t_i,r}(x)-\nabla_{x'}v^2_{t_i,r}(x')}\rho^{2,\ve}_{t_i,r}.\end{aligned}$$
Estimates \eqref{tu_25}, \eqref{tu_27}, \eqref{tu_000}, and \eqref{u_01000}, and the definition of the parabolic defect measures imply that, after applying H\"older's inequality and Young's inequality, if $m\geq 1$, for $C=C(m,U,T)>0$,
$$\begin{aligned} \abs{\textrm{Err}^{1,1}_{i,\ve}}  \leq &  C\omega(\abs{t_{i+1}-t_i};T)\int_{t_i}^{t_{i+1}}\int_\mathbb{R}\int_U q^1_r\dx\dxi\dr \\ & +C\omega(\abs{t_{i+1}-t_i};T)\int_{t_i}^{t_{i+1}}\int_U \abs{u^1}^{m-1}+\abs{u^1}^{m+1}\dx\dr,\end{aligned}$$
and, if $m\in(0,1)$, for $C=C(m,U,T)>0$,
$$\begin{aligned}\abs{\textrm{Err}^{1,1}_{i,\ve}}\leq & C\omega(\abs{t_{i+1}-t_i};T)\left(\abs{t_{i+1}-t_i}+\int_{t_i}^{t_{i+1}}\int_\mathbb{R}\int_U \abs{\xi}^{m-1}q^1_r\dx\dxi\dr\right) \\ & +\omega(\abs{t_{i+1}-t_i};T)\int_{t_i}^{t_{i+1}}\int_U\abs{u^1}^{m+1}\dx\dr.\end{aligned}$$
The difference in the cases $m\geq 1$ and $m\in(0,1)$ is that, for $m\in(0,1)$, the power $\abs{u^1}^{m-1}$ is not, in general, integrable.  Therefore, in the fast diffusion case, we absorb this term into a singular moment of the parabolic defect measure which is shown to be bounded in Proposition~\ref{aux_p} below.

The estimate for $\textrm{Err}^{2,1}_{i,\ve}$ and the errors $\{\textrm{Err}^{j,3}_{i,\ve}\}_{j\in\{1,2\}}$ are virtually identical.  We therefore conclude that, for each $j\in\{1,2\}$, for $C=C(m,U,T)>0$,
\begin{equation}\begin{aligned}\label{tu_28} \abs{\textrm{Err}^{j,1}_{i,\ve}}+\abs{\textrm{Err}^{j,3}_{i,\ve}}\leq & C\omega(\abs{t_{i+1}-t_i};T)\int_{t_i}^{t_{i+1}}\int_\mathbb{R}\int_U(1+\abs{\xi}^{(m-1)\wedge 0})q^j_r \\ & + C\omega(\abs{t_{i+1}-t_i};T)\int_{t_i}^{t_{i+1}}\int_U\abs{u^j}^{m+1}+\abs{u^j}^{(m-1)\vee 0}. \end{aligned}\end{equation}
The first term on the righthand side of \eqref{tu_28} is finite owing to Proposition~\ref{aux_p} below, and the second term is controlled using Lemma~\ref{interpolation} below, the boundedness of the domain, and H\"older's inequality.

\emph{$\{\textrm{Err}^{j,2}_{i,\ve}\}_{j\in\{1,2\}}$ and $\{\textrm{Err}^{j,4}_{i,\ve}\}_{j\in\{1,2\}}$}: We now treat the errors $\{\textrm{Err}^{j,2}_{i,\ve}\}_{j\in\{1,2\}}$ from \eqref{tu_14} and $\{\textrm{Err}^{j,4}_{i,\ve}\}_{j\in\{1,2\}}$ from \eqref{tu_20}.   For the case of $\textrm{Err}^{1,2}_{i,\ve}$, taking the absolute value and bounding $v_{t_i,r}$, $\sgn$, $\varphi_\beta$, and $v_{0,t_i}$ in $L^\infty(U\times[t_i,t_{i+1}])$, for $C=C(U,T)>0$,
$$\abs{\textrm{Err}^{1,2}_{i,\ve}}\leq C\int_{t_i}^{t_{i+1}}\int_{\mathbb{R}^3}\int_{U^3}\left(p^1_r+q^1_r\right)\rho^{1,\ve}_{t_i,r}\abs{\partial_\xi\Pi^{x,\xi}_{r,r-t_i}-\partial_{\xi'}\Pi^{x',\xi'}_{r,r-t_i}}\abs{\partial_\eta\rho^{2,\ve}_{t_i,r}}.$$
It then follows from estimates \eqref{tu_26}, \eqref{tu_000}, and \eqref{u_01000} that, for $C=C(U,T)>0$,
$$\abs{\textrm{Err}^{1,2}_{i,\ve}}\leq C\omega(\abs{t_{i+1}-t_i};T)\int_{t_i}^{t_{i+1}}\int_\mathbb{R}\int_U\left(p^1_r+q^1_r\right).$$
Since the estimates for $\textrm{Err}^{2,2}_{i,\ve}$ and $\{\textrm{Err}^{j,4}_{i,\ve}\}_{j\in\{1,2\}}$ are virtually identical, we conclude that, for each $j\in\{1,2\}$, for $C=C(U,T)>0$,
\begin{equation}\label{tu_29} \abs{\textrm{Err}^{j,2}}+\abs{\textrm{Err}^{j,4}}\leq  C \omega(\abs{t_{i+1}-t_i};T)\int_{t_i}^{t_{i+1}}\int_\mathbb{R}\int_U\left(p^j_r+q^j_r\right).\end{equation}

Define the mesh of the partition
$$\abs{\mathcal{P}}:=\max_{i\in\{0,\ldots,N-1\}}\abs{t_{i+1}-t_i}.$$
It follows from \eqref{tu_28} and \eqref{tu_29} that, for $C=C(m,U,T)>0$,
\begin{equation}\label{tu_31}\begin{aligned} & \sum_{j=1}^2\sum_{i=0}^{N-1}\abs{\textrm{Err}^{j,1}_{i,\ve}}+\abs{\textrm{Err}^{j,2}_{i,\ve}}+\abs{2\textrm{Err}^{j,3}_{i,\ve}}+\abs{2\textrm{Err}^{j,4}_{i,\ve}} \\ &   \leq C\sum_{j=1}^2 \omega(\abs{\mathcal{P}};T)\int_0^T\int_\mathbb{R}\int_U(1+\abs{\xi}^{(m-1)\wedge 0})\left(p^j_r+q^j_r\right) \\ & \quad + C\sum_{j=1}^2\omega(\abs{\mathcal{P}};T)\int_0^T\int_U\left( \abs{u^j}^{m+1}+\abs{u^j}^{(m-1)\vee 0}\right). \end{aligned}\end{equation}
Lemma~\ref{interpolation} below, Proposition~\ref{aux_p} below, the boundedness of the domain, and H\"older's inequality imply that the righthand side of \eqref{tu_31} vanishes in the limit $\abs{\mathcal{P}}\rightarrow 0$.  It remains to analyze the error $\textrm{Err}^5_{i,\ve}$.

\emph{$\textrm{Err}^5_{i,\ve}$}:  The analysis of the final error term $\textrm{Err}^5_{i,\ve}$ defined in \eqref{tu_23} will be divided into three cases:  $m=1$, $m\in(2,\infty)$, or $m\in(0,1)\cup(1,2]$.  The simplest of these is the case $m=1$.  Indeed, if $m=1$, then it is immediate from \eqref{tu_23} that $\textrm{Err}^5_{i,\ve}=0$.

\textit{Case $m\in(2,\infty)$}:  The argument will proceed by a decomposition of the integral.  For each $\delta\in(0,1)$, fix a smooth function $K_\delta:\mathbb{R}\rightarrow[0,1]$ satisfying
$$K_\delta(\xi):=\left\{\begin{array}{ll} 1 & \textrm{if}\;\; 2\delta\leq \abs{\xi}\leq \frac{1}{\delta}, \\ 0 & \textrm{if}\;\;\abs{\xi}\leq \delta\;\;\textrm{or}\;\;\abs{\xi}\geq \frac{2}{\delta}.\end{array}\right.$$
Then, for each $\delta\in(0,1)$, for $\textrm{err}^5_{i,\ve}$ defined in \eqref{error_5}, we consider the decomposition of the integrand, for each $r\in[t_i,t_{i+1}]$,
\begin{equation}\label{tu_new}\begin{aligned}  \int_{\mathbb{R}^3}\int_{U^3}\left(\textrm{err}^5_{i,\ve}\right)v_{0,t_i}\varphi_\beta  = & 2\int_{\mathbb{R}^3}\int_{U^3}K_\delta(\xi)K_\delta(\xi')\left(\textrm{err}^5_{i,\ve}\right)v_{0,t_i}\varphi_\beta \\ &+ 2\int_{\mathbb{R}^3}\int_{U^3}(1-K_\delta(\xi)K_\delta(\xi'))\left(\textrm{err}^5_{i,\ve}\right)v_{0,t_i}\varphi_\beta.\end{aligned}\end{equation}
For the first term of \eqref{tu_new}, after multiplying the integrand by $\abs{\xi}^\frac{m-1}{2}\abs{\xi'}^\frac{m-1}{2}$ and its inverse, the integration by parts formula \eqref{equation_ibp}, the boundedness of $v_{0,t_i}$ and $\varphi_\beta$, and the definition of $K_\delta$ imply that, for
$$\zeta(\xi,\xi'):=K_\delta(\xi)K_\delta(\xi')\frac{\left(\abs{\xi}^\frac{m-1}{2}-\abs{\xi'}^\frac{m-1}{2}\right)^2}{\abs{\xi}^\frac{m-1}{2}\abs{\xi'}^\frac{m-1}{2}},$$
there exists $C=C(m,U,T)>0$ such that
\begin{equation}\begin{aligned}\label{tnw_1} & \abs{\int_{t_i}^{t_{i+1}}\int_{\mathbb{R}^3}\int_{U^3}K_\delta(\xi)K_\delta(\xi')\left(\textrm{err}^5_{i,\ve}\right)v_{0,t_i}\varphi_\beta} \\ & \leq C\int_{t_i}^{t_{i+1}}\int_\mathbb{R}\int_{U^3}\zeta(u^1,u^2)\abs{\nabla_x \left(u^1\right)^\frac{m+1}{2}}\overline{\rho}^{1,\ve}_{t_i,r}\abs{\nabla_{x'}\left(u^2\right)^\frac{m+1}{2}}\overline{\rho}^{2,\ve}_{t_i,r}.\end{aligned}\end{equation}
The local Lipschitz continuity of the map $\xi\mapsto\abs{\xi}^{\frac{m-1}{2}}$ away from the origin and \eqref{tnw_1} prove that, for a constant $C=C(m,U,T,\delta)>0$ that blows up as $\delta\rightarrow 0$,
\begin{equation}\begin{aligned}\label{tnw_2} &\abs{2\int_{t_i}^{t_{i+1}}\int_{\mathbb{R}^3}\int_{U^3}K_\delta(\xi)K_\delta(\xi')\left(\textrm{err}^5_{i,\ve}\right)v_{0,t_i}\varphi_\beta} \\ & \leq C\ve^2\int_{t_i}^{t_{i+1}}\int_\mathbb{R}\int_{U^3}\abs{\nabla_x \left(u^1\right)^\frac{m+1}{2}}\overline{\rho}^{1,\ve}_{t_i,r}\abs{\nabla_{x'}\left(u^2\right)^\frac{m+1}{2}}\overline{\rho}^{2,\ve}_{t_i,r}.\end{aligned}\end{equation}
Therefore, after applying H\"older's inequality and Young's inequality, the definition of the convolution kernel, the definition of the weight \eqref{c_w}, the boundedness of the domain, the definition of the parabolic defect measures, and \eqref{tnw_2} prove that, for $C=C(m,U,T,\delta)>0$,
\begin{equation}\label{tu_new_0} \abs{2\int_{t_i}^{t_{i+1}}\int_{\mathbb{R}^3}\int_{U^3}K_\delta(\xi)K_\delta(\xi')\left(\textrm{err}^5_{i,\ve}\right)v_{0,t_i}\varphi_\beta} \leq C\ve\int_{t_i}^{t_{i+1}}\int_\mathbb{R}\int_U\left(q^1_r+q^2_r\right).\end{equation}

For the second term on the righthand side of \eqref{tu_new}, we use the following inequality, which is a consequence of the mean value theorem,
$$\left(\abs{\xi}^{\frac{m-1}{2}}-\abs{\xi'}^{\frac{m-1}{2}}\right)^2\leq \abs{\frac{m-1}{2}}^2\left(\abs{\xi}^{m-3}+\abs{\xi'}^{m-3}\right)\abs{\xi-\xi'}^2.$$
Together with \eqref{tu_000}, after bounding the weight \eqref{c_w} and $\varphi_\beta$ in $L^\infty(U\times[0,T])$, this inequality implies that, for the kernel
$$\Psi^\ve:=\abs{\ve\nabla_x\left(\rho_{t_i,r}^{1,\ve} v^1_{t_i,r}\right)}\abs{\ve\nabla_{x'}\left(\rho_{t_i,r}^{2,\ve} v^2_{t_i,r}\right)},$$
for each $r\in [t_i,t_{i+1}]$, for $C=C(m,U,T)>0$,
\begin{equation}\label{tu_new_1}\begin{aligned} & \abs{\int_{\mathbb{R}^3}\int_{U^3}(1-K_\delta(\xi)K_\delta(\xi'))\left(\textrm{err}^5_{i,\ve}\right)v_{0,t_i}\varphi_\beta} \\ & \leq C\int_{\mathbb{R}^3}\int_{U^3}(1-K_\delta(\xi)K_\delta(\xi'))\left(\abs{\xi}^{m-3}+\abs{\xi}^{m-1}\right)\abs{\chi^1_r}\abs{\chi^2_r}\Psi^\ve \\ & \quad + C\int_{\mathbb{R}^3}\int_{U^3}(1-K_\delta(\xi)K_\delta(\xi'))\left(\abs{\xi'}^{m-3}+\abs{\xi'}^{m-1}\right)\abs{\chi^1_r}\abs{\chi^2_r}\Psi^\ve.\end{aligned}\end{equation}
For the first term of \eqref{tu_new_1}, after bounding $\abs{\chi^2_r}$ in $L^\infty(U\times\mathbb{R}\times[t_i,t_{i+1}])$, a computation similar to \eqref{tu_0101} proves that, for each $r\in[t_i,t_{i+1}]$, for $C=C(m,U,T)>0$,
\begin{equation}\begin{aligned}\label{tnw_3} & \limsup_{\ve\rightarrow 0}\int_{\mathbb{R}^3}\int_{U^3}(1-K_\delta(\xi)K_\delta(\xi'))\left(\abs{\xi}^{m-3}+\abs{\xi}^{m-1}\right)\abs{\chi^1_r}\abs{\chi^2_r}\Psi^\ve \\ & \leq C\int_\mathbb{R}\int_U \left(1-K_\delta\left(\Xi^{x,\xi}_{t_i,r}\right)^2\right)\left(\abs{\Xi^{x,\xi}_{r,t_i}}^{m-3}+\abs{\Xi^{x,\xi}_{t_i,r}}^{m-1}\right)\abs{\tilde{\chi}^1_r} \\ & \leq C  \int_\mathbb{R}\int_U \left(1-K_\delta\left(\xi\right)^2\right)\left(\abs{\xi}^{m-3}+\abs{\xi}^{m-1}\right)\abs{\chi^1_r}.\end{aligned}\end{equation}
The definition of $K_\delta$, $m\in(2,\infty)$, properties of the kinetic function, and the boundedness of the domain imply that, for each $r\in[t_i,t_{i+1}]$ and $\delta\in(0,\frac{1}{2})$, for $C=C(m,U,T)>0$,
\begin{equation}\begin{aligned}\label{tnw_33} & \int_\mathbb{R}\int_U \left(1-K_\delta\left(\xi\right)^2\right)\left(\abs{\xi}^{m-3}+\abs{\xi}^{m-1}\right)\abs{\chi^1_r} \\ & \leq  C\left(\int_{-2\delta}^{2\delta}\abs{\xi}^{m-3}+\int_{\{\abs{\xi}\geq \frac{1}{\delta}\}}\int_U\abs{\xi}^{m-1}\abs{\chi^1}\right) \\  & \leq  C\left(\delta^{m-2}+\int_U\left(\abs{u^1}^m-\delta^{-m}\right)_+\right).\end{aligned}\end{equation}
The second term of \eqref{tu_new_1} is treated identically to the first, by obtaining \eqref{tnw_3} and \eqref{tnw_33} with $\chi^1$ replaced by $\chi^2$.  Therefore, returning to \eqref{tu_new_1}, after integrating over $r\in[t_i,t_{i+1}]$, for $C=C(m,U,T)>0$,
\begin{equation}\begin{aligned}\label{tnw_4} & \limsup_{\ve\rightarrow 0}\abs{\int_{t_i}^{t_{i+1}}\int_{\mathbb{R}^3}\int_{U^3}(1-K_\delta(\xi)K_\delta(\xi'))\left(\textrm{err}^5_{i,\ve}\right)v_{0,t_i}\varphi_\beta} \\ & \leq C\left(\delta^{m-2}\abs{t_{i+1}-t_i}+\int_{t_i}^{t_{i+1}}\int_U\left(\abs{u^1}^m-\delta^{-m}\right)_++\left(\abs{u^2}^m-\delta^{-m}\right)_+\right).\end{aligned}\end{equation}

Returning to $\textrm{Err}^5_{i,\ve}$ from \eqref{tu_23}, it follows from \eqref{tu_new_0} and \eqref{tnw_4} that, for $C_1=C_1(m,U,T,\delta)>0$ and $C_2=C_2(m,U,T)>0$,
$$\begin{aligned}  \limsup_{\ve\rightarrow 0}\sum_{i=0}^{N-1}\abs{\textrm{Err}^5_{i,\ve}}  & \leq \limsup_{\ve\rightarrow 0}C_1\ve\int_0^T\int_\mathbb{R}\int_U\left(q^1_r+q^2_r\right) \\ & \quad + \liminf_{\delta\rightarrow 0}C_2\left(\delta^{m-2}T+\int_0^T\int_U\left(\abs{u^1}^m-\delta^{-m}\right)_++\left(\abs{u^2}^m-\delta^{-m}\right)_+\right).\end{aligned}$$
Therefore, Lemma~\ref{lem_bdry}, $m\in(2,\infty)$, property (i) of Definition~\ref{sol_def}, and the dominated convergence theorem prove that
\begin{equation}\label{tu_3300}\limsup_{\ve\rightarrow 0}\sum_{i=0}^{N-1}\abs{\textrm{Err}^5_{i,\ve}}=0.\end{equation}
This completes the proof in the case $m\in(2,\infty)$.

\textit{Case $m\in(0,1)\cup(1,2]$}:  We will first form a decomposition to exclude large velocities.  For each $M>2$, let $\zeta_M:\mathbb{R}\rightarrow[0,1]$ denote a smooth cutoff function satisfying
\begin{equation}\label{fast_zeta} \zeta_M(\eta):=\left\{\begin{array}{ll} 1 & \textrm{if}\;\;\abs{\eta}\leq M, \\ 0 & \textrm{if}\;\;\abs{\eta}\geq M+1,\end{array}\right.\end{equation}
and, for $M>2$, consider the decomposition
\begin{equation}\begin{aligned}\label{fast_zeta_1}   \textrm{Err}^5_{i,\ve} = &  2\int_{t_i}^{t_{i+1}}\int_{\mathbb{R}^3}\int_{U^3} \left(\textrm{err}^5_{i,\ve}\right)v_{0,t_i}\varphi_\beta\zeta_M \\ & +2 \int_{t_i}^{t_{i+1}}\int_{\mathbb{R}^3}\int_{U^3} \left(\textrm{err}^5_{i,\ve}\right)v_{0,t_i}\varphi_\beta(1-\zeta_M).    \end{aligned} \end{equation}
For the second term of \eqref{fast_zeta_1}, the definition of the convolution kernel and the characteristics imply that, whenever
$$\rho^\ve_{t_i,r}(x,y,\xi,\eta)\rho^\ve_{t_i,r}(x',y,\xi',\eta)(1-\zeta_M(\eta))\neq 0,$$
we have, for $c=c(T)>0$,
\begin{equation}\label{fast_zeta_2} \min\{\abs{\xi},\abs{\xi'}\}\geq c(M-\ve).\end{equation}
Since the map $\xi\mapsto\abs{\xi}^\frac{m-1}{2}$ is globally Lipschitz away from the origin it follows from \eqref{tu_000},  the boundedness of $\varphi_\beta$, $\abs{\chi^2}$, and $\abs{\chi^1}$, the definition of the weight \eqref{c_w}, and \eqref{fast_zeta_2} that, for $C=C(m,U,T)>0$,
$$\begin{aligned} & \abs{\int_{t_i}^{t_{i+1}}\int_{\mathbb{R}^3}\int_{U^3} \left(\textrm{err}^5_{i,\ve}\right)v_{0,t_i}\varphi_\beta(1-\zeta_M)} \\ & \leq C\int_{t_i}^{t_{i+1}}\int_{\mathbb{R}^3}\int_{U^3} \abs{\xi}\abs{\chi^1_r}\abs{\ve\nabla_x\left(\rho^{1,\ve}_{t_i,r}v^1_{t_i,r}\right)}\abs{\ve \nabla_{x'}\left(\rho^{2,\ve}_{t_i,r}v^2_{t_i,r}\right)}(1-\zeta_M) \\ & \quad +C\int_{t_i}^{t_{i+1}}\int_{\mathbb{R}^3}\int_{U^3} \abs{\xi'}\abs{\ve\nabla_x\left(\rho^{1,\ve}_{t_i,r}v^1_{t_i,r}\right)}\abs{\chi^2_r}\abs{\ve \nabla_{x'}\left(\rho^{2,\ve}_{t_i,r}v^2_{t_i,r}\right)}(1-\zeta_M). \end{aligned} $$
A computation similar to that leading from \eqref{tnw_3} to \eqref{tnw_4}, the definition of $\zeta_M$, and \eqref{fast_zeta_2} prove that, for each $M>2$, for $C=C(m,U,T)>0$ and $c=c(T)>0$,
\begin{equation}\begin{aligned} \label{fast_zeta_3} & \limsup_{\ve\rightarrow 0}\abs{\int_{t_i}^{t_{i+1}}\int_{\mathbb{R}^3}\int_{U^3} \left(\textrm{err}^5_{i,\ve}\right)v_{0,t_i}\varphi_\beta(1-\zeta_M)} \\ & \leq C\int_{t_i}^{t_{i+1}}\int_{U} \left(\abs{u^1}^2-\left(c(M-\ve)\right)^2\right)_++\left(\abs{u^2}^2-\left(c(M-\ve)\right)^2\right)_+. \end{aligned} \end{equation}
Proposition~\ref{aux_p} below and the dominated convergence theorem imply that the righthand side of \eqref{fast_zeta_3} vanishes in the limit $M\rightarrow\infty$.

For the first term of \eqref{fast_zeta_1}, we first observe that an approximation argument, Proposition~\ref{aux_log} below, and the integration by parts formula \eqref{equation_ibp} imply that, for the kernel
$$\tilde{\Psi}^\ve_M:=\overline{\rho}^{1,\ve}_{t_i,r}v^1_{t_i,r} \overline{\rho}^{2,\ve}_{t_i,r}v^2_{t_i,r}v_{0,t_i}\varphi_\beta\zeta_M,$$
for $C=C(m)>0$,
\begin{equation}\label{fast_1}\begin{aligned}& \int_{t_i}^{t_{i+1}}\int_{\mathbb{R}^3}\int_{U^3} \left(\textrm{err}^5_{i,\ve}\right)v_{0,t_i}\varphi_\beta\zeta_M \\ &= C\int_{t_i}^{t_{i+1}}\int_\mathbb{R}\int_{U^3}\psi(u_1,u_2)\abs{u^1}^{-\frac{1}{2}}\nabla \left(u^1\right)^{\left[\frac{m+1}{2}\right]}\cdot\abs{u^2}^{-\frac{1}{2}}\nabla\left(u^2\right)^{\left[\frac{m+1}{2}\right]}\tilde{\Psi}^\ve_M,\end{aligned}\end{equation}
where
$$\psi(\xi,\xi'):=\abs{\xi}^{\frac{2-m}{2}}\abs{\xi'}^{\frac{2-m}{2}}\left(\abs{\xi}^\frac{m-1}{2}-\abs{\xi'}^\frac{m-1}{2}\right)^2.$$

We will now construct a decomposition of \eqref{fast_1}, where we will need to exclude a neighborhood of zero to handle the singularity of the diffusion coefficient.  For each $\delta\in(0,1)$, let $\tilde{K}_\delta:\mathbb{R}\rightarrow\mathbb[0,1]$ denote a smooth cutoff function satisfying
$$\left\{\begin{array}{ll} \tilde{K}_\delta(\xi)=0 & \textrm{if}\;\;\abs{\xi}\leq \delta, \\ \tilde{K}_\delta(\xi)=1 & \textrm{if}\;\;\abs{\xi}\geq 2\delta,\end{array}\right.$$
and returning to \eqref{fast_1} form the decomposition, for each $(y,\eta)\in U\times\mathbb{R}$, for $C=C(m)>0$,
\begin{equation}\label{ffast_2}\begin{aligned}  &\int_{t_i}^{t_{i+1}}\int_{\mathbb{R}^3}\int_{U^3} \left(\textrm{err}^5_{i,\ve}\right)v_{0,t_i}\varphi_\beta\zeta_M \\ & = C\int_{t_i}^{t_{i+1}}\int_\mathbb{R}\int_{U^3} \psi^\delta(u_1,u_2)\abs{u^1}^{-\frac{1}{2}}\nabla \left(u^1\right)^{\left[\frac{m+1}{2}\right]}\cdot\abs{u^2}^{-\frac{1}{2}}\nabla\left(u^2\right)^{\left[\frac{m+1}{2}\right]}\tilde{\Psi}^\ve_M \\ &\quad +  C\int_{t_i}^{t_{i+1}}\int_\mathbb{R}\int_{U^3} \tilde{\psi}^\delta(u_1,u_2)\abs{u^1}^{-\frac{1}{2}}\nabla \left(u^1\right)^{\left[\frac{m+1}{2}\right]}\cdot\abs{u^2}^{-\frac{1}{2}}\nabla\left(u^2\right)^{\left[\frac{m+1}{2}\right]}\tilde{\Psi}^\ve_M,\end{aligned}\end{equation}
where, for each $\delta\in(0,1)$, $\psi^\delta,\tilde{\psi}^\delta:\mathbb{R}^2\rightarrow\mathbb{R}$ are defined by
\begin{equation}\label{ffast_7}\psi^\delta(\xi,\xi'):=\tilde{K}_\delta(\xi)\tilde{K}_\delta(\xi')\psi(\xi,\xi')\;\textrm{and}\;\tilde{\psi}^\delta(\xi,\xi'):=(1-\tilde{K}_\delta(\xi)\tilde{K}_\delta(\xi'))\psi(\xi,\xi').\end{equation}

For the first term of \eqref{ffast_2}, the definition of the characteristics, the definition of $\zeta_M$ in \eqref{fast_zeta}, and \eqref{ffast_7} imply that, whenever
$$\psi^\delta(\xi,\xi')\zeta_M(\eta)\neq 0,$$
we have, for $c=c(T)>0$,
$$\delta\leq \abs{\xi},\abs{\xi'}\leq c(M+\ve).$$
The local Lipschitz continuity of the map $\xi\in\mathbb{R}\mapsto\abs{\xi}^{\frac{m-1}{2}}$ away from zero therefore implies that, for $C=C(m,M,\delta)>0$,
$$\abs{\psi^\delta(\xi,\xi')}\leq C\abs{\xi-\xi'}^2.$$
It then follows from Proposition~\ref{aux_log} below, H\"older's inequality, Young's inequality, and the definition of the parabolic defect measures that, since $\varphi\geq \varphi_\beta$, for $C=C(m,U,T,M,\delta)>0$,
\begin{equation}\label{ffast_3}\begin{aligned} &\abs{\int_{t_i}^{t_{i+1}}\int_\mathbb{R}\int_{U^3} \psi^\delta(u_1,u_2)\abs{u^1}^{-\frac{1}{2}}\nabla \left(u^1\right)^{\left[\frac{m+1}{2}\right]}\cdot\abs{u^2}^{-\frac{1}{2}}\nabla\left(u^2\right)^{\left[\frac{m+1}{2}\right]}\tilde{\Psi}^\ve_M}
 \\ & \leq C\ve^2\abs{\int_{t_i}^{t_{i+1}}\int_\mathbb{R}\int_{U^3} \abs{u^1}^{-\frac{1}{2}}\nabla \left(u^1\right)^{\left[\frac{m+1}{2}\right]}\cdot\abs{u^2}^{-\frac{1}{2}}\nabla\left(u^2\right)^{\left[\frac{m+1}{2}\right]}\tilde{\Psi}^\ve_M} \\ & \leq C\ve(\int_{t_i}^{t_{i+1}}\int_\mathbb{R}\int_U\abs{\xi}^{-1}q^1_r\varphi\dx\dxi\dr)^\frac{1}{2}(\int_{t_i}^{t_{i+1}}\int_\mathbb{R}\int_U\abs{\xi'}^{-1}q^2_r\varphi\dxp\dxip\dr)^\frac{1}{2} \\ & \leq C\ve\left(1+\norm{u_0^1}^2_{L^2(U)}+\norm{u^2_0}^2_{L^2(U)}\right).\end{aligned}\end{equation}

For the second term of \eqref{ffast_2}, it follows from \eqref{tu_000} and \eqref{fast_zeta} that, for $C_1=C_1(T,M)>0$,
\begin{equation}\label{fast_2}\zeta_M(\eta)\overline{\rho}^{1,\ve}_{t_i,r}\cdot \overline{\rho}^{2,\ve}_{t_i,r}\neq 0\;\;\textrm{implies that}\;\;\abs{u^1-u^2}\leq C_1\ve.\end{equation}
Observe that if $\max\{{\abs{\xi},\abs{\xi'}\}}\leq 2C_1\ve$, then a direct computation yields, for $C=C(T,M)>0$ depending on $C_1$,
\begin{equation}\label{fast_3}\psi(\xi,\xi')\leq\abs{\xi}^\frac{m}{2}\abs{\xi'}^\frac{2-m}{2}+2\abs{\xi}^\frac{1}{2}\abs{\xi'}^\frac{1}{2}+\abs{\xi}^\frac{2-m}{2}\abs{\xi'}^\frac{m}{2}\leq C\ve.\end{equation}
If $\abs{\xi-\xi'}\leq C_1\ve$ with $\max\{{\abs{\xi},\abs{\xi'}\}}\geq 2C_1\ve$, assume without loss of generality that $\abs{\xi}\leq\abs{\xi'}$.  A Lipschitz estimate proves that, for $C=C(m,T,M)>0$ depending on $C_1$,
$$\left(\abs{\xi}^\frac{m-1}{2}-\abs{\xi'}^\frac{m-1}{2}\right)^2\leq C\min\{\abs{\xi},\abs{\xi'}\}^{m-3}\abs{\xi-\xi'}^2 \leq C\abs{\xi}^{m-3}\abs{\xi-\xi'}^2\leq C\abs{\xi}^{m-3}\ve^2.$$
Therefore, if $\abs{\xi-\xi'}\leq C_1\ve$ with $\max\{{\abs{\xi},\abs{\xi'}\}}\geq 2C_1\ve$,  which implies that $\abs{\xi}\geq \frac{1}{2}\abs{\xi'}$, for $C=C(m,T,M)>0$ depending on $C_1$,
\begin{equation}\label{fast_4}\psi(\xi,\xi')\leq C\abs{\xi}^\frac{2-m}{2}\abs{\xi'}^\frac{2-m}{2}\abs{\xi}^{m-3}\ve^2\leq C\abs{\xi}^{-1}\ve^2\leq C\ve.\end{equation}

The definition of $\tilde{\psi}^\delta$ in \eqref{ffast_7}, the fact that $\tilde{\psi}^\delta(\xi,\xi')=0$ on the set $\{\xi=\xi'\}$, and the fact that the set
$$\{u^1\neq u^2\}\subset\left(\{u^1\neq 0\}\cup\{u^2\neq 0\}\right),$$
imply with \eqref{fast_2}, \eqref{fast_3}, and \eqref{fast_4} that, for each $\delta\in(0,1)$, there exists $C=C(m,U,T,M)>0$ for which
\begin{equation}\label{ffast_8}\begin{aligned} &\abs{\int_{t_i}^{t_{i+1}}\int_\mathbb{R}\int_{U^3} \tilde{\psi}^\delta(u_1,u_2)\abs{u^1}^{-\frac{1}{2}}\nabla \left(u^1\right)^{\left[\frac{m+1}{2}\right]}\cdot\abs{u^2}^{-\frac{1}{2}}\nabla\left(u^2\right)^{\left[\frac{m+1}{2}\right]}\tilde{\Psi}^\ve_M} \\ & \leq C\ve\abs{\int_{t_i}^{t_{i+1}}\int_\mathbb{R}\int_{U}\int_{U^\delta_1}\abs{u^1}^{-\frac{1}{2}}\nabla \left(u^1\right)^{\left[\frac{m+1}{2}\right]}\cdot\abs{u^2}^{-\frac{1}{2}}\nabla\left(u^2\right)^{\left[\frac{m+1}{2}\right]}\tilde{\Psi}^\ve_M} \\ & \quad +C\ve\abs{\int_{t_i}^{t_{i+1}}\int_\mathbb{R}\int_{U^\delta_2}\int_U \abs{u^1}^{-\frac{1}{2}}\nabla \left(u^1\right)^{\left[\frac{m+1}{2}\right]}\cdot\abs{u^2}^{-\frac{1}{2}}\nabla\left(u^2\right)^{\left[\frac{m+1}{2}\right]}\tilde{\Psi}^\ve_M},\end{aligned}\end{equation}
where, for each $j\in\{1,2\}$,
\begin{equation}\label{ffast_9}U^\delta_j:=\left\{\;x\in U\;|\;0<\abs{u^j(x)}<2\delta\;\right\}.\end{equation}
Estimate \eqref{ffast_8} relies on the fact that $\tilde{\psi}^\delta(u^1,u^2)$ vanishes on the compliment of the set $(U^\delta_1\cup U^\delta_2)$.

For the first term of \eqref{ffast_8}, the definition of the parabolic defect measures, definition of the convolution kernel, H\"older's inequality, and Young's inequality prove that, since $\varphi\geq\varphi_\beta$, for $C=C(m,U,T,M)>0$,
$$\begin{aligned} & \limsup_{\ve\rightarrow 0}\abs{\ve\int_{t_i}^{t_{i+1}}\int_\mathbb{R}\int_{U}\int_{U^\delta_1} \abs{u^1}^{-\frac{1}{2}}\nabla \left(u^1\right)^{\left[\frac{m+1}{2}\right]}\cdot\abs{u^2}^{-\frac{1}{2}}\nabla\left(u^2\right)^{\left[\frac{m+1}{2}\right]}\tilde{\Psi}^\ve_M}\\ & \leq C(\int_{t_i}^{t_{i+1}}\int_\mathbb{R}\int_{U^\delta_1}\abs{\xi}^{-1}q^1_r\varphi\dx\dxi\dr)^\frac{1}{2}(\int_{t_i}^{t_{i+1}}\int_\mathbb{R}\int_{U}\abs{\xi'}^{-1}q^2_r\varphi \dx\dxi\dr)^\frac{1}{2}.\end{aligned}$$
Proposition~\ref{aux_log} below, \eqref{ffast_9}, and the dominated convergence theorem imply that the righthand side of \eqref{ffast_8} vanishes in the limit $\delta\rightarrow 0$.  The second term of \eqref{ffast_8} is handled identically, after swapping the roles of $\chi^1$ and $\chi^2$.

Returning to \eqref{ffast_2}, it follows from estimates \eqref{ffast_3} and \eqref{ffast_8} that, uniformly for $M>2$, after passing first to the limit $\ve\rightarrow 0$ and then to the limit $\delta\rightarrow 0$,
$$\limsup_{\ve\rightarrow 0}\abs{\int_{t_i}^{t_{i+1}}\int_{\mathbb{R}^3}\int_{U^3} \left(\textrm{err}^5_{i,\ve}\right)v_{0,t_i}\varphi_\beta\zeta_M}=0.$$
Returning to \eqref{fast_zeta_1}, after passing to the limit $M\rightarrow\infty$ in \eqref{fast_zeta_3}, we conclude that
\begin{equation}\label{tu_3400}\limsup_{\ve\rightarrow 0}\sum_{i=0}^{N-1}\abs{\textrm{Err}^5_{i,\ve}}= 0,\end{equation}
which completes the analysis of the error terms.

\textbf{Step 6:  The limit $\ve\rightarrow 0$.}  For $p_m=\left(\frac{m+1}{m}\wedge 2\right)$, from Lemma~\ref{lem_bdry} we have, for each $j\in\{1,2\}$, for $c_m=\frac{2m}{m+1}$,
\begin{equation}\label{tu_lem}u^{[m]}\in L^{p_m}([0,T];W^{1,p_m}_0(U))\;\;\textrm{with}\;\;\nabla \left(u^j\right)^{[m]}=c_m\abs{u^j}^{\frac{m-1}{2}}\nabla \left(u^j\right)^{\left[\frac{m+1}{2}\right]}.\end{equation}
Returning to \eqref{tu_24}, the integration by parts formula \eqref{equation_ibp} and \eqref{tu_lem} imply that
\begin{equation}\label{tu_36}\begin{aligned} & \int_{t_i}^{t_{i+1}}\int_{\mathbb{R}^2}\int_{U^2}m\abs{\xi}^{m-1}\chi^1_r \nabla_x\left(\rho^{1,\ve}_{t_i,r}v^1_{t_i,r}\right)\cdot \left(\tilde{\sgn}^\ve_{t_i,r}-2\tilde{\chi}^{2,\ve}_{t_i,r}\right)\nabla_y\left(v_{0,t_i}\varphi_\beta\right) \\ & = -\int_{t_i}^{t_{i+1}}\int_\mathbb{R}\int_U\left(\int_U \nabla(u^1)^{[m]}\overline{\rho}^{1,\ve}_{t_i,r}v^1_{t_i,r}\right)\cdot\left(\tilde{\sgn}^\ve_{t_i,r}-2\tilde{\chi}^{2,\ve}_{t_i,r}\right)\nabla_y\left(v_{0,t_i}\varphi_\beta\right).\end{aligned}\end{equation}
Using Fubini's theorem,
\begin{equation}\begin{aligned}\label{tuc_0} & \int_{t_i}^{t_{i+1}}\int_\mathbb{R}\int_U\left(\int_U \nabla(u^1)^{[m]}\overline{\rho}^{1,\ve}_{t_i,r}v^1_{t_i,r}\right)\cdot\left(\tilde{\sgn}^\ve_{t_i,r}-2\tilde{\chi}^{2,\ve}_{t_i,r}\right)\nabla_y\left(v_{0,t_i}\varphi_\beta\right) \\ & =\int_{t_i}^{t_{i+1}}\int_U \nabla\left(u^1\right)^{[m]}v_{t_i,r}\cdot\left(\int_\mathbb{R}\int_U\left(\tilde{\sgn}^\ve_{t_i,r}-2\tilde{\chi}^{2,\ve}_{t_i,r}\right)\overline{\rho}^{1,\ve}_{t_i,r}\nabla_y\left(v_{0,t_i}\varphi_\beta\right) \right).\end{aligned}\end{equation}
It follows from the definition of $\overline{\rho}^{1,\ve}_{t_i,r}$ that the inner integral has a convolution structure.  Precisely, for the convolution kernels $\rho^\ve_d$ and $\rho^\ve_1$ defining \eqref{tu_00}, and for
$$f^\ve:=\left(\tilde{\sgn}^\ve_{t_i,r}-2\tilde{\chi}^{2,\ve}_{t_i,r}\right)\nabla_y\left(v_{0,t_i}\varphi_\beta\right),$$
for each $(x,\xi,r)\in U\times\mathbb{R}\times[t_i,t_{i+1}]$, the inner integral of \eqref{tuc_0} has the form
\begin{equation}\label{tuc_1}  \left(f^\ve*\rho^\ve_1\rho^\ve_d\right)\left(x,\Pi^{x,u^1(x,r)}_{r,r-t_i}\right) =\int_\mathbb{R}\int_Uf^\ve(y,\eta)\rho^\ve_d(y-x)\rho^\ve_1(\eta-\Pi^{x,u^1(x,r)}_{r,r-t_i})\dy\deta.\end{equation}
Since the characteristics preserve the sign of the velocity variable, it follows from the boundedness of the $\sgn$ and kinetic functions, \eqref{tu_0101}, and \eqref{tu_010101} that, for each $p\in[1,\infty)$, strongly in $L^p_{\textrm{loc}}(\overline{U}\times\mathbb{R}\times[t_i,t_{i+1}])$,
\begin{equation}\label{tuc_2}\lim_{\ve\rightarrow 0}\left(\tilde{\sgn}^\ve_{t_i,r}(y,\eta)-2\tilde{\chi}^{2,\ve}_{t_i,r}(y,\eta)\right)=\sgn(\Xi^{y,\eta}_{t_i,r})-2\chi^2(y,\Xi^{y,\eta}_{t_i,r},r).\end{equation}
It follows from the triangle inequality, the fact that convolution does not increase $L^p$-norms, the inverse property of the characteristics \eqref{c_inverse}, and \eqref{tuc_2} that, for each $p\in[1,\infty)$, strongly in $L^p_{\textrm{loc}}(\overline{U}\times\mathbb{R}\times[t_i,t_{i+1}])$,
\begin{equation}\begin{aligned}\label{tuc_3}  & \lim_{\ve\rightarrow 0}\left(\left(\tilde{\sgn}^\ve_{t_i,r}-2\tilde{\chi}^{2,\ve}_{t_i,r}\right)\nabla_y\left(v_{0,t_i}\varphi_\beta\right)*\rho^\ve_1\rho^\ve_d\right)\left(x,\Pi^{x,u^1(x,r)}_{r,r-t_i}\right) \\ & =\left.\left(\sgn(\Xi^{x,\xi}_{t_i,r})-2\chi^2(x,\Xi^{x,\xi}_{t_i,r},r)\right)\nabla_x\left(v_{0,t_i}(x)\varphi_\beta(x)\right)\right|_{\xi=\Pi^{x,u^1(x,r)}_{r,r-t_i}} \\ & = \left(\sgn(u^1(x,r))-2\chi^2(x,u^1(x,r),r)\right)\nabla_x\left(v_{0,t_i}(x)\varphi_\beta(x)\right).\end{aligned}\end{equation}
Returning to \eqref{tuc_0}, it follows from \eqref{tuc_1}, \eqref{tuc_3}, and Lemma~\ref{lem_bdry} that
$$\begin{aligned} &\lim_{\ve\rightarrow 0}\int_{t_i}^{t_{i+1}}\int_\mathbb{R}\int_U\left(\int_U \nabla(u^1)^{[m]}\overline{\rho}^{1,\ve}_{t_i,r}v^1_{t_i,r}\right)\cdot\left(\tilde{\sgn}^\ve_{t_i,r}-2\tilde{\chi}^{2,\ve}_{t_i,r}\right)\nabla_y\left(v_{0,t_i}\varphi_\beta\right) \\ &  =\int_{t_i}^{t_{i+1}}\int_U\left(\nabla\left(u^1\right)^{[m]}v_{t_i,r}\right)\cdot\left(\sgn\left(u^1\right)-2\chi^2(x,u^1,r)\right) \nabla\left(v_{0,t_i}\varphi_\beta\right).\end{aligned}$$
Finally, properties of the kinetic function prove that
\begin{equation}\begin{aligned}\label{tu_3600} & \int_{t_i}^{t_{i+1}}\int_U\left(\nabla\left(u^1\right)^{[m]}v_{t_i,r}\right)\cdot\left(\sgn\left(u^1\right)-2\chi^2(x,u^1,r)\right) \nabla\left(v_{0,t_i}\varphi_\beta\right)\\ & =\int_{t_i}^{t_{i+1}}\int_U\left(\nabla\left(u^1\right)^{[m]}v_{t_i,r}\right)\cdot \sgn\left(u^1\right)\left(1-2\textbf{1}_{\left\{\abs{u^1}<\abs{u^2}\right\}}\right)\nabla\left(v_{0,t_i}\varphi_\beta\right). \end{aligned}\end{equation}
Therefore, after swapping the roles of $j\in\{1,2\}$ to obtain the analogue of \eqref{tu_36} to \eqref{tu_3600} for the second term of \eqref{tu_24}, estimates \eqref{tu_31}, \eqref{tu_3300}, and \eqref{tu_3400} show that, for $C=C(m,U,T)>0$,
\begin{equation}\label{tu_39}\begin{aligned} & \left.\int_{\mathbb{R}}\int_U\abs{\chi^1(y,\eta,r)-\chi^2(y,\eta,r)}v_{0,r}\varphi_\beta\right|_{r=0}^T \\ & \leq -\sum_{i=0}^{N-1}\int_{t_i}^{t_{i+1}}\int_U\left(\nabla\left(u^1\right)^{[m]}v_{t_i,r}\right)\cdot \sgn\left(u^1\right)\left(1-2\textbf{1}_{\left\{\abs{u^1}<\abs{u^2}\right\}}\right)\nabla\left(v_{0,t_i}\varphi_\beta\right) \\ & \quad -\sum_{i=0}^{N-1}\int_{t_i}^{t_{i+1}}\int_U\left(\nabla\left(u^2\right)^{[m]}v_{t_i,r}\right)\cdot \sgn\left(u^2\right)\left(1-2\textbf{1}_{\left\{\abs{u^2}<\abs{u^1}\right\}}\right)\nabla\left(v_{0,t_i}\varphi_\beta\right) \\ &\quad +\sum_{j=1}^2 C\omega(\abs{\mathcal{P}};T)\left(\int_0^T\int_\mathbb{R}\int_U(1+\abs{\xi}^{(m-1)\wedge 0})\left(p^j_r+q^j_r\right)+\int_0^T\int_U\left( \abs{u^j}^{m+1}+\abs{u^j}^{(m-1)\vee 0}\right)\right).\end{aligned} \end{equation}
This completes the analysis of the limit $\ve\rightarrow 0$.

\textbf{Step 7:  The limit $\abs{\mathcal{P}}\rightarrow 0$.}  We fix a nested sequence of partitions $\left\{\mathcal{P}_k\subset[0,T]\setminus\mathcal{N}\right\}_{k=1}^\infty$ satisfying $\abs{\mathcal{P}_k}\rightarrow 0$ as $k\rightarrow\infty$.  For each $k\in\mathbb{N}$, we will write $\{t_i^k\}_{i=0}^{N_k}$ for the elements of the partition $\mathcal{P}_k$, where $N_k$ denotes the total number of elements.

For the final two terms of \eqref{tu_39}, since the constants are independent of the partition, it follows from Lemma~\ref{interpolation} below, Proposition~\ref{aux_p} below, H\"older's inequality, and the boundedness of the domain that
\begin{equation}\begin{aligned}\label{tu_0043}  \lim_{k\rightarrow\infty} \left(\sum_{j=1}^2\right. &C\omega(\abs{\mathcal{P}_k};T)\int_0^T\int_\mathbb{R}\int_U(1+\abs{\xi}^{(m-1)\wedge 0})\left(p^j_r+q^j_r\right)\dx\dxi\dr \\  & \left.+ \sum_{j=1}^2C\omega(\abs{\mathcal{P}_k};T)\int_0^T\int_U \abs{u^j}^{m+1}+\abs{u^j}^{(m-1)\vee 0}\dx\dr\right)=0.\end{aligned}\end{equation}
For the first term of \eqref{tu_39}, for each $k\in\mathbb{N}$ and $i\in\{0,\ldots,N_k-1\}$,
\begin{equation}\label{nu_0041}\begin{aligned} & \int_{t_i^k}^{t_{i+1}^k}\int_U\left(\nabla\left(u^1\right)^{[m]}v_{t_i^k,r}\right)\cdot \sgn\left(u^1\right)\left(1-2\textbf{1}_{\left\{\abs{u^1}<\abs{u^2}\right\}}\right)\nabla\left(v_{0,t_i^k}\varphi_\beta\right) \\ & = \int_{t_i^k}^{t_{i+1}^k}\int_U\nabla\left(u^1\right)^{[m]}\left(v_{t_i^k,r}-1\right)\cdot \sgn\left(u^1\right)\left(1-2\textbf{1}_{\left\{\abs{u^1}<\abs{u^2}\right\}}\right)\nabla\left(v_{0,t_i^k}\varphi_\beta\right) \\ & \quad +\int_{t_i^k}^{t_{i+1}^k}\int_U\nabla\left(u^1\right)^{[m]}\cdot \sgn\left(u^1\right)\left(1-2\textbf{1}_{\left\{\abs{u^1}<\abs{u^2}\right\}}\right)\nabla\left(\left(v_{0,t_i^k}-v_{0,r}\right)\varphi_\beta\right) \\ & \quad +\int_0^T\int_U\left(\nabla\left(u^1\right)^{[m]}\right)\cdot \sgn\left(u^1\right)\left(1-2\textbf{1}_{\left\{\abs{u^1}<\abs{u^2}\right\}}\right)\nabla\left(v_{0,r}\varphi_\beta\right).\end{aligned}\end{equation}
For the first two terms on the righthand side of \eqref{nu_0041}, it follows from the definition of the weight \eqref{c_w} and the regularity of the coefficients that, after bounding the $\sgn$ and indicator functions in $L^\infty(U\times[0,T])$, for each $k\in\mathbb{N}$ and $i\in\{0,\ldots,N_k-1\}$, for $C=C(U,T)>0$,
\begin{equation}\begin{aligned}\label{tu_0042}& \abs{\int_{t_i^k}^{t_{i+1}^k}\int_U\nabla\left(u^1\right)^{[m]}\left(v_{t_i^k,r}-1\right)\cdot \sgn\left(u^1\right)\left(1-2\textbf{1}_{\left\{\abs{u^1}<\abs{u^2}\right\}}\right)\nabla\left(v_{0,t_i^k}\varphi_\beta\right)} \\ & + \abs{\int_{t_i^k}^{t_{i+1}^k}\int_U\nabla\left(u^1\right)^{[m]}\cdot \sgn\left(u^1\right)\left(1-2\textbf{1}_{\left\{\abs{u^1}<\abs{u^2}\right\}}\right)\nabla\left(\left(v_{0,t_i^k}-v_{0,r}\right)\varphi_\beta\right)} \\ &  \quad \leq  C \omega\left(\abs{t^k_{i+1}-t^k_i};T\right)\int_{t^k_i}^{t^k_{i+1}}\int_U \abs{\nabla\left(u^1\right)^{[m]}}\left(\varphi_\beta+\abs{\nabla\varphi_\beta}\right).\end{aligned}\end{equation}
Therefore, after summing over $i\in\{0,\ldots,N_k-1\}$ and passing to the limit $k\rightarrow\infty$, the choice of the partitions, Lemma~\ref{lem_bdry} below, \eqref{tu_39}, \eqref{tu_0043}, \eqref{nu_0041}, and \eqref{tu_0042} prove that
\begin{equation}\label{nu_41}\begin{aligned} & \left.\int_{\mathbb{R}}\int_U\abs{\chi^1(x,\xi,r)-\chi^2(x,\xi,r)}v_{0,r}\varphi_\beta\dy\deta\right|_{r=0}^T \\ & \leq -\int_0^T\int_U\left(\nabla\left(u^1\right)^{[m]}\right)\cdot \sgn\left(u^1\right)\left(1-2\textbf{1}_{\left\{\abs{u^1}<\abs{u^2}\right\}}\right)\nabla_y\left(v_{0,r}\varphi_\beta\right) \\ & \quad -\int_0^T\int_U\left(\nabla\left(u^2\right)^{[m]}\right)\cdot \sgn\left(u^2\right)\left(1-2\textbf{1}_{\left\{\abs{u^2}<\abs{u^1}\right\}}\right)\nabla_y\left(v_{0,r}\varphi_\beta\right).\end{aligned}\end{equation}
The righthand side of \eqref{nu_41} can be simplified.  For $\psi^1,\psi^2\in\C^\infty_c(U)$, observe the distributional equality
$$\begin{aligned} & \nabla\abs{\abs{\psi^1}-\abs{\psi^2}} \\ & =  \sgn(\psi^1)\left(1-2\textbf{1}_{\{\abs{\psi^1}<\abs{\psi^2}\}}\right)\nabla\psi^1 +\sgn(\psi^2)\left(1-2\textbf{1}_{\{\abs{\psi^2}<\abs{\psi^1}\}}\right)\nabla\psi^2.\end{aligned}$$
Therefore, it follows from \eqref{nu_41}, Lemma~\ref{lem_bdry} below, the density of smooth functions, and the monotonicity of the map $\xi\mapsto\xi^{[m]}$ that, after integrating by parts,
\begin{equation}\label{tu_40} \left.\int_{\mathbb{R}}\int_U\abs{\chi^1(y,\eta,r)-\chi^2(y,\eta,r)}v_{0,r}\varphi_\beta\dy\deta\right|_{r=0}^T \leq \int_0^T\int_U\abs{\abs{\left(u^1\right)^{[m]}}-\abs{\left(u^2\right)^{[m]}}}\Delta\left(v_{0,r}\varphi_\beta\right).\end{equation}
This completes the time-splitting argument.

\textbf{Step 8:  The limit $\beta\rightarrow 0$ and the conclusion.}  For each $\beta\in(0,1)$, recall the definition $\varphi_\beta=\textbf{1}_\beta\varphi$ from \eqref{tu_3} for $\varphi$ satisfying \eqref{u_aux} and $\textbf{1}_\beta$ satisfying \eqref{tu_1} and \eqref{tu_2}.  The Lipschitz continuity of $\varphi$ up to the boundary, Gilbarg and Trudinger \cite[Theorem~6.14]{GilbargTrudinger}, and the definition of $U_\beta$ from \eqref{tu_01} prove that, for $C=C(U)>0$ independent of $\beta\in(0,1)$,
$$\sup_{U\setminus U_\beta}\varphi \leq C\beta\;\;\textrm{and}\;\;\sup_U\abs{\nabla\varphi}\leq C.$$
Therefore, for $\textbf{1}_\beta$ defined in \eqref{tu_1} satisfying estimates \eqref{tu_2},  there exists $C=C(U,T)>0$ such that, for each $\beta\in(0,1)$ and $r\in[0,T]$,
$$\abs{\textbf{1}_\beta\Delta\left(\varphi v_{0,r}\right) -\Delta\left(\varphi_\beta v_{0,r}\right)}\leq \frac{C}{\beta}\textbf{1}_{U\setminus U_\beta}\;\;\textrm{on}\;\;U.$$
Returning to \eqref{tu_40}, it follows that, for $C=C(U,T)>0$,
\begin{equation}\label{tu_42}\begin{aligned} & \int_0^T\int_U\abs{\abs{\left(u^1\right)^{[m]}}-\abs{\left(u^2\right)^{[m]}}}\Delta\left(v_{0,r}\varphi_\beta\right) \\ & \leq  \int_0^T\int_U \abs{\abs{\left(u^1\right)^{[m]}}-\abs{\left(u^2\right)^{[m]}}}\textbf{1}_\beta\Delta\left(v_{0,r}\varphi\right)+ \frac{C}{\beta}\int_0^T\int_{U\setminus U_\beta}\abs{\abs{\left(u^1\right)^{[m]}}-\abs{\left(u^2\right)^{[m]}}}.\end{aligned}\end{equation}
For the second term of \eqref{tu_42}, Lemma~\ref{lem_bdry} below and the density of smooth functions imply following an explicit calculation that, for $C=C(U,T)>0$,
\begin{equation}\label{tu_00042}\frac{C}{\beta}\int_0^T\int_{U\setminus U_\beta}\abs{\abs{\left(u^1\right)^{[m]}}-\abs{\left(u^2\right)^{[m]}}}\leq C\int_0^T\int_{U\setminus U_\beta}\abs{\nabla\left(u^1\right)^{[m]}-\nabla\left(u^2\right)^{[m]}}.\end{equation}
The dominated convergence theorem and Lemma~\ref{lem_bdry} below prove that the righthand side of \eqref{tu_00042} vanishes in the limit $\beta\rightarrow 0$.  Therefore, returning to \eqref{tu_40},
\begin{equation}\label{tu_43} \left.\int_{\mathbb{R}}\int_U\abs{\chi^1(x,\xi,r)-\chi^2(x,\xi,r)}v_{0,r}\varphi\right|_{r=0}^T \leq \int_0^T\int_U\abs{\abs{\left(u^1\right)^{[m]}}-\abs{\left(u^2\right)^{[m]}}}\Delta\left(v_{0,r}\varphi\right).\end{equation}

In order to conclude, observe that definition \eqref{c_w} and \eqref{u_aux} imply that there exists $t_*\in(0,\infty)\setminus\mathcal{N}$ such that, for each $r\in[0,t_*]$,
\begin{equation}\label{tu_44}\Delta\left(v_{0,r}\varphi\right)\leq 0\;\;\textrm{in}\;\;U.\end{equation}
Therefore, since $T\in[0,\infty)\setminus\mathcal{N}$ was arbitrary, we conclude from \eqref{tu_43} that, for each $T\in[0,t_*]\setminus\mathcal{N}$,
\begin{equation}\label{nu_45}\left.\int_{\mathbb{R}}\int_U\abs{\chi^1(x,\xi,r)-\chi^2(x,\xi,r)}v_{0,r}\varphi\dy\deta\right|_{r=0}^T\leq 0.\end{equation}
The general statement now follows by induction.

Let $T\in(0,\infty)\setminus\mathcal{N}$ be arbitrary.  Then using \eqref{c_w}, there exists $\tilde{t}=\tilde{t}(T)\leq t_*\in(0,\infty)$ such that, whenever $0\leq s\leq t\leq T$ satisfy $\abs{s-t}\leq \tilde{t}$,
\begin{equation}\label{tu_46}\Delta\left(v_{s,t}\varphi\right)\leq 0\;\;\textrm{in}\;\;U.\end{equation}
Since $\tilde{t}\leq t_*$, it follows from the definition of the weight \eqref{c_w}, \eqref{nu_45}, and \eqref{tu_46} that, for $C=C(T)>0$,
\begin{equation}\label{tu_47}\norm{\left(u^1-u^2\right)\varphi}_{L^\infty([0,\tilde{t}];L^1(U))}\leq C\norm{\left(u^1_0-u^2_0\right)\varphi}_{L^1(U)}.\end{equation}
For the inductive hypothesis, for simplicity alone assume that $\{k\tilde{t}\}_{k\in\mathbb{N}}\subset[0,\infty)\setminus\mathcal{N}$ and suppose that, for some $k\in\mathbb{N}$ with $k\tilde{t}<T$, there exists $C=C(T)>0$ such that
\begin{equation}\label{tu_48}\norm{\left(u^1-u^2\right)v_{0,t}\varphi}_{L^\infty([0,k\tilde{t}];L^1(U))}\leq C\norm{\left(u^1_0-u^2_0\right)\varphi}_{L^1(U)}.\end{equation}
A repetition of the arguments leading to \eqref{tu_47} on the interval $[k\tilde{t},(k+1)\tilde{t}\wedge T]$, replacing $v_{0,r}$ with $v_{k\tilde{t},r}$, yields the inequality
$$\norm{\left(u^1-u^2\right)v_{k\tilde{t},t}\varphi}_{L^\infty([k\tilde{t},(k+1)\tilde{t}\wedge T];L^1(U))}\leq \norm{\left(u^1(\cdot,k\tilde{t})-u^2(\cdot,k\tilde{t})\right)\varphi}_{L^1(U)}.$$
The regularity of the coefficients, the definition of the weight \eqref{c_w}, and the inductive hypothesis \eqref{tu_48} therefore imply that, for $C=C(T)>0$,
$$\norm{\left(u^1-u^2\right)\varphi}_{L^\infty([0,(k+1)\tilde{t}\wedge T];L^1(U))}\leq C\norm{\left(u^1_0-u^2_0)\right)\varphi}_{L^1(U)}.$$
Since \eqref{tu_47} is the base case, this completes the inductive argument and therefore the proof.\end{proof}

The following corollary establishes uniform estimates for pathwise kinetic solutions in $L^1$, and proves that the nonnegativity of the initial data is preserved by the solution.  The proof is a consequence of the proof of Theorem~\ref{thm_unique}.

\begin{cor}\label{aux_cor} Let $u_0\in L^2(U)$.  Suppose that $u$ is a pathwise kinetic solution of \eqref{intro_eq} in the sense of Definition~\ref{sol_def} with initial data $u_0$.  Then, for $\varphi$ satisfying \eqref{u_aux}, for each $T>0$ there exists $C=C(T)>0$ such that
$$\norm{u\varphi}_{L^\infty([0,T]);L^1(U))}\leq C\norm{u_0\varphi}_{L^1(U)}.$$
Furthermore, if $u_0\in L^2_+(U)$, then $u\in L^\infty_{\textrm{loc}}([0,\infty);L^2_+(U))$.
\end{cor}

\begin{proof}  Let $u_0\in L^2(U)$ be arbitrary, and let $u$ be a pathwise kinetic solution of \eqref{intro_eq} with initial data $u_0$.  For the first claim, since the zero function, with vanishing entropy and parabolic defect measures, is a pathwise kinetic solution of \eqref{intro_eq} with vanishing initial data, Theorem~\ref{thm_unique} implies that, for each $T>0$, for $C=C(T)>0$,
$$\begin{aligned} & \norm{u\varphi}_{L^\infty([0,T];L^1(U))}=\norm{\left(u-0\right)\varphi}_{L^\infty([0,T];L^1(U))}\\ & \leq C\norm{\left(u_0-0\right)\varphi}_{L^1(U)}=C\norm{u_0\varphi}_{L^1(U)},\end{aligned}$$
which completes the proof.

For the second claim, let $u_0\in L^2_+(U)$ be arbitrary and suppose that $u$ is a pathwise kinetic solution of \eqref{intro_eq} with initial data $u_0$, kinetic function $\chi$, and exceptional set $\mathcal{N}$.  Following the reasoning leading from \eqref{tu_7} to \eqref{tu_15} with the $\sgn$ function replaced by its negative part
$$\sgn_-:=(\sgn\wedge 0),$$
repeating the estimates of the error terms \eqref{tu_31}, and passing to the limit first with respect to $\ve\in(0,1)$ and second with respect to $\abs{\mathcal{P}}$ yields, after applying the integration by parts formula \eqref{equation_ibp} and using Lemma~\ref{lem_bdry} below, for each $T\in[0,\infty)\setminus\mathcal{N}$, 
$$\left.\int_\mathbb{R}\int_U\chi(x,\xi,r)\sgn_-(\xi)v_{0,r}\right|_{r=0}^T\leq -\int_0^T\int_\mathbb{R}\int_U \nabla u^{[m]}\sgn_-(u)\nabla(v_{0,r}\varphi).$$
Therefore, using the monotonicity of the map $\xi\mapsto\xi^{[m]}$, after integrating by parts, for each $T\in[0,\infty)\setminus\mathcal{N}$,
\begin{equation}\label{path_L1_1}\left.\int_\mathbb{R}\int_U\chi(x,\xi,r)\sgn_-(\xi)v_{0,r}\right|_{r=0}^T\leq \int_0^T\int_\mathbb{R}\int_U \abs{u^{[m]}\wedge 0}\Delta(v_{0,r}\varphi).\end{equation}
It follows from \eqref{tu_44} that, for all $T>0$ sufficiently small, the righthand side of \eqref{path_L1_1} is negative.  Therefore, for all $T>0$ sufficiently small, the definition of the kinetic function and the nonnegativity of the initial data imply that
$$0\leq \int_\mathbb{R}\int_U\chi(x,\xi,T)\sgn_-(\xi)v_{0,T}\varphi \leq \int_\mathbb{R}\int_U\overline{\chi}(u_0(x),\xi)\varphi=0.$$
The statement for general $T\in[0,\infty)\setminus\mathcal{N}$ follows by induction, as in the proof of Theorem~\ref{thm_unique}.  The $L^2$-estimate is true for general initial data $u_0\in L^2(U)$ and is the case $\delta=1$ in Proposition~\ref{aux_p} below.  \end{proof}

We conclude this section with a few auxiliary estimates that are necessary for the proof of uniqueness in the case $m\in(0,1)\cup(1,2]$.  In the arguments, we will repeatedly use the following estimate, which is immediate from the Poincar\'e inequality and the Dirichlet boundary conditions.  This is a simplified version of the interpolation estimate \cite[Lemma~4.2]{FehrmanGess}, which would be needed in the case of non-vanishing boundary data.

\begin{lem}\label{interpolation}  Suppose that $z\in\C^\infty_c(U)$ with $z^{\left[\nicefrac{m+1}{2}\right]}\in H^1_0(U)$. Then, for $C=C(U)>0$,
$$\norm{z}^{m+1}_{L^{m+1}(U)}=\norm{z^{\left[\frac{m+1}{2}\right]}}^2_{L^2(U)} \leq C\norm{\nabla z^{\left[\frac{m+1}{2}\right]}}_{L^2(U;\mathbb{R}^d)}^2.$$
\end{lem}

\begin{proof}  Suppose that $z\in\C^\infty_c(U)$ with
$$z^{\left[\frac{m+1}{2}\right]}\in H^1_0(U).$$
The first equality is immediate from the definitions.  The second inequality is the Poincar\'e inequality applied to the function $z^{[\frac{m+1}{2}]}$ on the domain $U$.  This completes the proof.  \end{proof}

In the following two estimates, we will obtain bounds for singular moments of the entropy and kinetic defect measures in a neighborhood of the origin.  The first estimate is particularly relevant in the fast diffusion case $m\in(0,1)$, and it is used in the proof of Lemma~\ref{lem_bdry} below to establish the regularity of the signed power $u^{[m]}$.

\begin{prop}\label{aux_p}  Let $\delta\in(0,1]$ and $u_0\in L^2(U)$ be arbitrary.  Suppose that $u$ is a pathwise kinetic solution of \eqref{intro_eq} in the sense of Definition~\ref{sol_def} with initial data $u_0$.  Then, for each $T>0$, there exists $C=C(m,U,T,\delta)>0$ such that
$$\norm{u}^{1+\delta}_{L^\infty([0,T];L^{1+\delta}(U))}+\int_0^T\int_\mathbb{R}\int_U \abs{\xi}^{\delta-1}\left(p+q\right)\leq C\left(\norm{u_0}^{1+\delta}_{L^{1+\delta}(U)}+\norm{u_0}_{L^2(U)}^{2\frac{m+\delta}{m+1}}\right).$$
\end{prop}

\begin{proof}  Fix $\delta\in(0,1]$ and $u_0\in L^2(U)$ be arbitrary.  Suppose that $u$ is a pathwise kinetic solution of \eqref{intro_eq} with initial data $u_0$, kinetic function $\chi$, entropy and parabolic defect measures $p$ and $q$, and exceptional set $\mathcal{N}$.

An approximation argument implies that $\xi\mapsto\xi^{[\delta]}$ is an admissible test function.  Therefore, for each $T\in[0,\infty)\setminus\mathcal{N}$, after applying equation \eqref{transport_equation} on the interval $[0,T]$ to the test function $\rho(\xi)=\xi^{[\delta]}$, and using the identity
$$\Pi^{x,\xi}_{t,t}=\xi \exp\left(\sum_{k=1}^nf_k(x)z^k_{t,0}\right)=\xi v_{0,t}(x),$$
which is immediate from \eqref{c_char} and \eqref{c_w}, we have
\begin{equation}\label{aux_p_0} \left.\int_{\mathbb{R}}\int_U \chi_r\xi^{[\delta]}v_{0,r}^{1+\delta}\right|_{r=0}^T+\delta\int_0^T\int_\mathbb{R}\int_U\abs{\xi}^{\delta-1}\left(p+q\right)v_{0,r}^{1+\delta} =\int_0^T\int_\mathbb{R}\int_U m\xi^{[m+\delta-1]}\chi\Delta v_{0,r}^{1+\delta}.\end{equation}
The integration by parts formula \eqref{equation_ibp} implies that
$$ \int_0^T\int_\mathbb{R}\int_U m\xi^{[m+\delta-1]}\chi\Delta v_{0,r}^{1+\delta} =-\frac{2m}{m+1}\int_0^T\int_U \abs{u}^{\frac{m+\delta}{2}}\abs{u}^{\frac{\delta-1}{2}}\nabla u^{\left[\frac{m+1}{2}\right]}\cdot\nabla v_{0,r}^{1+\delta}.$$
Therefore, H\"older's inequality, Young's equality, and the definition of the parabolic defect measure imply that, for $C_1=C_1(m)>0$,
\begin{equation}\label{aux_p_1} \abs{\int_0^T\int_\mathbb{R}\int_U m\xi^{[m+\delta-1]}\chi\Delta v_{0,r}^{1+\delta}}\leq C_1\norm{\nabla v_0^{1+\delta}}_{L^\infty(U\times[0,T];\mathbb{R}^d)}\left(\int_0^T\int_U\abs{u}^{m+\delta}+\int_0^T\int_\mathbb{R}\int_U\abs{\xi}^{\delta-1}q \right).\end{equation}

We will first specialize to the case $\delta=1$.  It follows from part (i) of Definition~\ref{sol_def}, Lemma~\ref{interpolation}, and \eqref{aux_p_1} that, for $C_2=C_2(m,U)>0$,
$$\abs{\int_0^T\int_\mathbb{R}\int_U m\xi^{[m]}\chi\Delta v_{0,r}^2}\leq C_2\norm{\nabla v_0^2}_{L^\infty(U\times[0,T];\mathbb{R}^d)}\int_0^T\int_\mathbb{R}\int_U q.$$
The definition of the weight \eqref{c_w}, the continuity of the noise, and the regularity of the coefficients imply that there exists $t_*=t_*(m,U)\in(0,\infty)\setminus\mathcal{N}$ such that
$$\inf_{(x,r)\in U\times[0,t_*]} v_{0,r}^{2}-C_2\norm{\nabla v_0^2}_{L^\infty(U\times[0,T];\mathbb{R}^d)}\geq \frac{1}{2}.$$
Returning to \eqref{aux_p_0} with $\delta=1$, it follows from the definition of the kinetic function and the definition of the weight \eqref{c_w} that, for each $T\in[0,t_*]\setminus\mathcal{N}$, for $C=C(m,U,t_*)>0$,
$$\norm{u(\cdot,T)}^2_{L^2(U)}+\int_0^T\int_\mathbb{R}\int_U \left(p+q\right) \leq C\norm{u_0}^2_{L^2(U)}.$$
The statement for general $T\in[0,\infty)\setminus\mathcal{N}$ follows by induction, as in the proof of Theorem~\ref{thm_unique}.  Precisely, for each $T>0$, there exists $C=C(m,U,T)>0$ such that
\begin{equation}\label{aux_p_3}\norm{u}^2_{L^\infty([0,T];L^2(U))}+\int_0^T\int_\mathbb{R}\int_U (p+q)\leq C\norm{u_0}^2_{L^2(U)}.\end{equation}

We now consider the case of $\delta\in(0,1)$.  Returning to \eqref{aux_p_1}, H\"older's inequality, part (i) of Definition~\ref{sol_def}, and Lemma~\ref{interpolation} imply that, for $C=C(m,U,T)>0$,
\begin{equation}\label{aux_p_4}\int_0^T\int_U\abs{u}^{m+\delta}\leq C\left(\int_0^T\int_U\abs{u}^{m+1}\right)^{\frac{m+\delta}{m+1}}\leq C\left(\int_0^T\int_\mathbb{R}\int_Uq\right)^{\frac{m+\delta}{m+1}}.\end{equation}
The definition of the weight \eqref{c_w}, the continuity of the noise, and the regularity of the coefficients imply that there exists $t_*=t_*(m,\delta)\in(0,\infty)\setminus\mathcal{N}$ such that
\begin{equation}\label{aux_p_5}\inf_{(x,r)\in U\times[0,t_*]} \delta v_{0,r}^{1+\delta}-C_2\norm{\nabla v_0^{1+\delta}}_{L^\infty(U\times[0,T];\mathbb{R}^d)}\geq \frac{\delta}{2}.\end{equation}
Therefore, returning to \eqref{aux_p_0}, it follows from \eqref{aux_p_1}, \eqref{aux_p_3}, \eqref{aux_p_4}, and \eqref{aux_p_5} together with the definition of the kinetic function and the definition of the weight \eqref{c_w} that, for each $T\in[0,t_*]\setminus\mathcal{N}$, for $C=C(m,U,t_*,\delta)>0$,
$$ \norm{u(\cdot,T)}^{1+\delta}_{L^{1+\delta}(U)}+\int_0^T\int_\mathbb{R}\int_U\abs{\xi}^{\delta-1}\left(p+q\right) \leq C\left(\norm{u_0}^{1+\delta}_{L^{1+\delta}(U)}+\norm{u_0}_{L^2(U)}^{2\frac{m+\delta}{m+1}}\right).$$
The statement for general $T\in[0,\infty)\setminus\mathcal{N}$ follows by induction, as in the proof of Theorem~\ref{thm_unique} and \eqref{aux_p_3} above.  That is, for each $T>0$, there exists $C=C(m,U,T,\delta)>0$ such that
$$ \norm{u}^{1+\delta}_{L^\infty([0,T];L^{1+\delta}(U))}+\int_0^T\int_\mathbb{R}\int_U \abs{\xi}^{\delta-1}(p+q)\dx\dxi\dr \leq C\left(\norm{u_0}^{1+\delta}_{L^{1+\delta}(U)}+\norm{u_0}_{L^2(U)}^{2\frac{m+\delta}{m+1}}\right).$$
This completes the proof.  \end{proof}

\begin{remark}  We observe that by repeating the arguments of Lemma~\ref{lem_bdry} the conclusion of Proposition~\ref{aux_p} can be improved to show that, for each $T>0$, for $C=C(m,U,T,\delta)>0$,
\begin{equation}\label{aux_p_7}\norm{u}^{1+\delta}_{L^\infty([0,T];L^{1+\delta}(U))}+\int_0^T\int_\mathbb{R}\int_U \abs{\xi}^{\delta-1}\left(p+q\right)\dx\dxi\dr\leq C\norm{u_0}^{1+\delta}_{L^{1+\delta}(U)}.\end{equation}
For this, for each $\delta\in(0,1)$, it is only necessary to prove that, for each $T>0$,
$$u^{\left[\frac{m+\delta}{2}\right]}\in L^2([0,T];H^1_0(U))\;\;\textrm{with}\;\;\nabla u^{\left[\frac{m+\delta}{2}\right]}= \frac{m+\delta}{m+1}\abs{u}^{\frac{\delta-1}{2}}\nabla u^{\left[\frac{m+1}{2}\right]}.$$
It then follows from the Poincar\'e inequality that \eqref{aux_p_4} can be estimated by the singular moment of the parabolic defect measure appearing on the lefthand side of \eqref{aux_p_0}.  The time-splitting argument proves that this term can be absorbed, which yields \eqref{aux_p_7} and completes the proof.  Lemma~\ref{lem_bdry} proves this in the special case $\delta=m$, for diffusion exponents $m\in(0,1]$.  \end{remark}

The final proposition of this section establishes a bound for a singular moment.  This estimate effectively applies regularity in the case $\delta=0$ from Proposition~\ref{aux_p}.  Informally, this implies the local $L^2$-integrability of $\nabla u^{\frac{m}{2}}$, which was used in the proof of uniqueness for diffusion exponents $m\in(0,1)\cup (1,2]$.  We require here the nonnegativity of the initial data.  The estimate is false, in general, for signed initial data \cite[Remark~4.8]{FehrmanGess}.

\begin{prop}\label{aux_log}  Let $u_0\in L^2_+(U)$ be arbitrary.  Suppose that $u$ is a pathwise kinetic solution of \eqref{intro_eq} in the sense of Definition~\ref{sol_def} with initial data $u_0$.  Then, for every $\psi\in\C^\infty_c(U)$, for each $T>0$, there exists $C=C(m,U,T,\psi)>0$ such that
$$\int_0^T\int_\mathbb{R}\int_U \abs{\xi}^{-1}\left(p+q\right)\psi\leq C\left(1+\norm{u_0}_{L^2(U)}^2\right).$$
\end{prop}

\begin{proof}  Let $u_0\in L^2_+(U)$ and $\psi\in\C^\infty_c(U)$ be arbitrary.  Suppose that $u$ is a pathwise kinetic solution of \eqref{intro_eq} with initial data $u_0$, kinetic function $\chi$, entropy and parabolic defect measures $p$ and $q$, and exceptional set $\mathcal{N}$.  An approximation argument, which relies crucially on the nonnegativity of the initial data and therefore the nonnegativity of the solution by Corollary~\ref{aux_cor}, implies that $\rho(x,\xi)=\log(\xi)\psi(x)$ is an admissible test function.

Therefore, for each $T\in[0,T]\setminus\mathcal{N}$, we apply equation \eqref{transport_equation} with $\rho_0(x,\xi)=\log(\xi)\psi(x)$ on the interval $[0,T]$, for which $\rho_{0,r}$ defined in \eqref{c_f} satisfies
$$\begin{aligned}\rho_{0,r}(x,\xi) & =\log(\Pi^{x,\xi}_{r,r})\psi(x)v_{0,r}(x) \\ & =\log(\xi)\psi(x)v_{0,r}(x)+\left(\sum_{k=1}^nf_k(x)z^k_{r,0}\right)\psi(x)v_{0,r}(x),\end{aligned}$$
which follows from \eqref{c_char} and \eqref{c_w}, to obtain
\begin{equation}\begin{aligned}\label{alog_0}  & \left.\int_\mathbb{R}\int_U\chi(x,\xi,r)\log(\xi v_{0,r})\psi v_{0,r}\right|_{r=0}^T+\int_0^T\int_\mathbb{R}\int_U \abs{\xi}^{-1}(p+q)\psi v_{0,r} \\ & =  \int_0^T\int_\mathbb{R}\int_U m\abs{\xi}^{m-1}\chi\left(\log(\xi)\Delta\left(\psi v_{0,r}\right)+\sum_{k=1}^n\Delta(f_k \psi v_{0,r})z^k_{r,0}\right).  \end{aligned} \end{equation}

For the first term of \eqref{alog_0}, the integrability of the logarithm in a neighborhood of the origin, definition of the kinetic function, the definition of the weight \eqref{c_w}, and Proposition~\ref{aux_p} imply that, for $C=C(m,U,T,\psi)>0$,
\begin{equation}\begin{aligned} \label{alog_1}\abs{\left.\int_\mathbb{R}\int_U\chi(x,\xi,r)\log(\xi v_{0,r})\psi v_{0,r}\dx\dxi\right|_{r=0}^T} & \leq C\left(1+\norm{u}^2_{L^\infty([0,T];L^2(U))}\right) \\ & \leq C\left(1+\norm{u_0}^2_{L^2(U)}\right).\end{aligned}\end{equation}
For the righthand side of \eqref{alog_0}, the integrability of $\xi\mapsto \abs{\xi}^{m-1}\log(\xi)$ in a neighborhood of the origin, the definition of the kinetic function, the definition of the weight \eqref{c_w}, Lemma~\ref{interpolation}, and Proposition~\ref{aux_p} imply that, for $C=C(m,U,T,\psi)>0$,
\begin{equation}\begin{aligned}\label{alog_2} & \abs{\int_0^T\int_\mathbb{R}\int_U m\abs{\xi}^{m-1}\chi\left(\log(\xi)\Delta\left(\psi v_{0,r}\right)+\sum_{k=1}^n\Delta(f_k \psi v_{0,r})z^k_{r,0}\right)} \\ & \leq  C\left(1+\int_0^T\int_U\abs{u}^{m+1}\right) \leq C\left(1+\norm{u_0}^2_{L^2(U)}\right).\end{aligned}\end{equation}
Returning to \eqref{alog_0}, estimates \eqref{alog_1} and \eqref{alog_2} complete the proof.  \end{proof}

\section{Stable estimates and existence}

In this section, we will prove the existence of pathwise kinetic solutions by establishing stable estimates for the regularized equation, defined for each $\eta\in(0,1)$ and $\ve\in(0,1)$, 
\begin{equation}\label{char_reg} \left\{\begin{array}{ll} \partial_t u^{\eta,\ve}=\Delta \left(u^{\eta,\ve}\right)^{[m]}+\eta\Delta u^{\eta,\ve}+\sum_{k=1}^nf_k(x)u^{\eta,\ve}\dot{z}^{k,\ve}_t & \textrm{in}\;\;U\times(0,\infty), \\ u^{\eta,\ve}=u_0 & \textrm{on}\;\; U\times\{0\}, \\ u^{\eta,\ve}=0 & \textrm{on}\;\;\partial U\times(0,\infty),\end{array}\right.\end{equation}
where the smooth path $z^\ve$ is defined in \eqref{convolve_noise}.  The well-posedness of \eqref{char_reg} was previously established in Proposition~\ref{reg_exists}.

Let $u_0\in L^2(U)$, $\eta\in(0,1)$, and $\ve\in(0,1)$ be fixed but arbitrary.  For the solution $u^{\eta,\ve}$ from Proposition~\ref{char_reg}, we denote the kinetic function
\begin{equation}\label{char_kinetic_new} \chi^{\eta,\ve}(x,\xi,t):=\overline{\chi}(u^{\eta,\ve}(x,t),\xi),\end{equation}
we will write $p^{\eta,\ve}$ and $q^{\eta,\ve}$ for the corresponding entropy and parabolic defect measures.  It was shown in Proposition~\ref{char_keq} that the kinetic function $\chi^{\eta,\ve}$ is a distributional solution of the equation
$$\partial_t\chi^{\eta,\ve}= m\abs{\xi}^{m-1}\Delta \chi^{\eta,\ve}+\eta\Delta\chi^{\eta,\ve}-\partial_\xi \chi^{\eta,\ve}\sum_{k=1}^n \xi f_k(x)\dot{z}^\ve_t+\partial_\xi\left(p^{\eta,\ve}+q^{\eta,\ve}\right),$$
with initial data $\overline{\chi}(u_0,\xi)$.  Finally, in Proposition~\ref{c_eq}, for the smooth forward characteristic defined in \eqref{c_forward}, the transported kinetic function, defined for $s\geq 0$ by $\tilde{\chi}^{\eta,\ve}_{s,t}(x,\xi):=\chi^{\eta,\ve}(x,\Xi^{x,\xi,\ve}_{s,t},t)$, was shown to satisfy equation \eqref{c_eq_1} defined for a class of test function transported by the smooth backward characteristic \eqref{c_backward}.  We will now establish stable estimates for the solutions $\{u^{\eta,\ve}\}_{\eta,\ve\in(0,1)}$ which allow us to pass to the limit $\eta,\ve\rightarrow 0$.

We first obtain a stable estimate for the solution $L^2(U)$ and for the entropy and parabolic defect measures.  The proof is omitted, since it is virtually identical to the proof of Proposition~\ref{aux_p} in the case $\delta=1$.
\begin{prop}\label{char_stable_estimates} Let $\eta\in(0,1)$, $\ve\in(0,1)$, and $u_0\in L^2(U)$ be arbitrary.  For the solution $u^{\eta,\ve}$ of \eqref{char_reg} with initial data $u_0$, let $\chi^{\eta,\ve}$ denote the corresponding kinetic function and let $p^{\eta,\ve}$ and $q^{\eta,\ve}$ denote the corresponding entropy and parabolic defect measures.  For each $T>0$, there exists $C=C(m,U,T)>0$ such that
$$\norm{u^{\eta,\ve}}^2_{L^\infty([0,T];L^2(U))}+\int_0^T\int_\mathbb{R}\int_U \left(p^{\eta,\ve}+q^{\eta,\ve}\right)\dx\dxi\dr\leq C\norm{u_0}^2_{L^2(U)}.$$
\end{prop}

We observe that the transport of the kinetic function corresponds to a simple transformation of the original equation \eqref{char_reg}.  For each $\eta\in(0,1)$ and $\ve\in(0,1)$, consider the transformation
\begin{equation}\label{char_sol_transform} \tilde{u}^{\eta,\ve}:=v^\ve_{0,t} u^{\eta,\ve}.\end{equation}
It is immediate by considering test functions $\rho_0\in \C^\infty_c(U)$ that are independent of $\xi\in\mathbb{R}$ in Proposition~\ref{c_eq} that the transformed solutions \eqref{char_sol_transform} are distributional solutions of the equation
\begin{equation}\label{char_sol_transform_eq} \left\{\begin{array}{l} \partial_t\tilde{u}^{\eta,\ve}=v^\ve_{0,t}\Delta\left( \left(\tilde{u}^{\eta,\ve}\right)^{[m]}\left(v^\ve_{0,t}\right)^{-m}\right)+\eta v^\ve_{0,t}\Delta \left(\tilde{u}^{\eta,\ve}\left(v^\ve_{0,t}\right)^{-1}\right), \\ \tilde{u}^{\eta,\ve}=u_0, \\ \tilde{u}^{\eta,\ve}=0, \end{array}\right.\end{equation}
for the domains defined by \eqref{char_reg}, where this equation implicitly uses the fact that the weights \eqref{c_weight} are positive.  The reason for introducing the transformed solutions \eqref{char_sol_transform} is that we cannot expect to obtain a uniform control of the time derivatives $\{\partial_t u^{\eta,\ve}\}_{\eta,\ve\in(0,1)}$ in the singular limit $\ve\rightarrow 0$.  However, as will be seen in Proposition~\ref{exist_time} below, the transformation effectively cancels the irregularity introduced by the noise, and the derivatives $\{\partial_t\tilde{u}^{\eta,\ve}\}_{\eta,\ve\in(0,1)}$ can be estimated uniformly in a negative Sobolev space.

The following corollary of Proposition~\ref{char_stable_estimates} is based on the work of Ebmeyer \cite{Ebmeyer}.  For each $s\in(0,1)$, the space $W^{s,m+1}(U)$ denotes the fractional Sobolev space defined by the norm, for $\psi\in\C^\infty_c(U)$,
$$\norm{\psi}_{W^{s,m+1}(U)}:=\norm{\psi}_{L^{m+1}(U)}+\left(\int_{U\times U}\frac{\abs{\psi(x)-\psi(y)}^{m+1}}{\abs{x-y}^{d+s(m+1)}}\dx\dy\right)^{\frac{1}{m+1}}.$$
In the precise formulation appearing here, with only small modifications, the proof follows from  \cite[Proposition~C.3]{FehrmanGess}, Lemma~\ref{interpolation}, and the definition of the weight \eqref{c_weight}.  We therefore omit the details.

\begin{cor}\label{exist_fractional}  Let $\eta\in(0,1)$, $\ve\in(0,1)$, and $u_0\in L^2(U)$ be arbitrary.  For $u^{\eta,\ve}$ the solution of \eqref{char_reg} from Proposition~\ref{char_reg} with initial data $u_0$, for each $T>0$ and $s\in\left(0,\frac{2}{m+1}\wedge 1\right)$, there exists $C=C(m,U,T,s)>0$ such that
$$\norm{u^{\eta,\ve}}^{m+1}_{L^{m+1}([0,T];W^{s,m+1}(U))}\leq C\norm{u_0}_{L^2(U)}^2.$$
In particular, for $\tilde{u}^{\eta,\ve}$ defined in \eqref{char_sol_transform}, for each $T>0$ and $s\in\left(0,\frac{2}{m+1}\wedge 1\right)$, there exists $C=C(m,U,T,s)>0$ such that
$$\norm{\tilde{u}^{\eta,\ve}}^{m+1}_{L^{m+1}([0,T];W^{s,m+1}(U))}\leq C\norm{u_0}_{L^2(U)}^2.$$
\end{cor}

We now obtain estimates for the family of time derivatives $\{\partial_t\tilde{u}^{\eta,\ve}\}_{\eta,\ve\in(0,1)}$.  The proof of the following proposition is essentially a consequence of equation \eqref{char_sol_transform_eq}, Lemma~\ref{lem_bdry}, and H\"older's inequality.  In what follows, for a Holder exponent $p\in(1,\infty)$, we will write $W^{-1,p}(U)$ for the dual space of $W^{1,p}_0(U)$.

\begin{prop}\label{exist_time}  Let $\eta\in(0,1)$, $\ve\in(0,1)$, and $u_0\in L^2(U)$.  For the solution $u^{\eta,\ve}$ from Proposition~\ref{char_reg} and $\tilde{u}^{\eta,\ve}$ defined in \eqref{char_sol_transform}, for each $T>0$ there exists $C=C(m,U,T)>0$ such that, for
$$p_m:=\left(\frac{m+1}{m}\wedge 2\right)\;\;\textrm{and the H\"older exponent}\;\;q_m=\frac{p_m}{p_m-1},$$
if $m\in[1,\infty)$,
$$\norm{\partial_t\tilde{u}^{\eta,\ve}}_{L^{p_m}([0,T];W^{-1,q_m}(U))}\leq C(\norm{u_0}^{\frac{2}{p_m}}_{L^2(U)}+\eta^\frac{1}{2}\norm{u_0}_{L^2(U)}),$$
and, if $m\in(0,1)$,
$$\norm{\partial_t\tilde{u}^{\eta,\ve}}_{L^{p_m}([0,T];W^{-1,q_m}(U))}\leq C(\norm{u_0}^{\frac{1+m}{2}}_{L^{1+m}(U)}+\norm{u_0}_{L^2(U)}^{\frac{2m}{m+1}}+\eta^\frac{1}{2}\norm{u_0}_{L^2(U)}).$$
\end{prop}

\begin{proof}  Let $\eta\in(0,1)$, $\ve\in(0,1)$, and $u_0\in L^2(U)$ be arbitrary.  Let $u^{\eta,\ve}$ denote the solution from Proposition~\ref{char_reg}, and let $\tilde{u}^{\eta,\ve}$ be as in \eqref{char_sol_transform}.  Let $T>0$ be fixed but arbitrary.  It follow from the definition of the weight \eqref{c_weight}, equation \eqref{char_sol_transform_eq}, H\"older's inequality, and the Poincar\'e inequality that, for each $t\in[0,T]$ and $\psi\in \C^\infty_c(U)$, for $C=C(U,T)>0$,
$$ \abs{\int_U \partial_t\tilde{u}^{\eta,\ve}(\cdot,t)\psi\dx}\leq C\left(\norm{\nabla \left(u^{\eta,\ve}\right)^{[m]}(\cdot,t)}_{L^{p_m}(U;\mathbb{R}^d)}+\norm{\eta\nabla u^{\eta,\ve}}_{L^{p_m}(U;\mathbb{R}^d)}\right)\norm{\nabla\psi}_{L^{q_m}(U;\mathbb{R}^d)}.$$
Therefore, for each $t\in[0,T]$, for $C=C(U,T)>0$,
$$ \norm{\partial_t\tilde{u}^{\eta,\ve}(\cdot,t)}_{W^{-1,q_m}(U)} \leq C\left(\norm{\nabla \left(u^{\eta,\ve}\right)^{[m]}(\cdot,t)}_{L^{p_m}(U;\mathbb{R}^d)}+\norm{\eta\nabla u^{\eta,\ve}(\cdot,t)}_{L^{p_m}(U;\mathbb{R}^d)}\right).$$
After integrating in time and applying H\"older's inequality, since $p_m\leq 2$, for $C=C(U,T)>0$,
$$\begin{aligned}& \norm{\partial_t\tilde{u}^{\eta,\ve}}_{L^{p_m}\left([0,T];W^{-1,q_m}(U)\right)} \\ & \leq C\left(\norm{\nabla \left(u^{\eta,\ve}\right)^{[m]}(\cdot,t)}_{L^{p_m}(U\times [0,T];\mathbb{R}^d)}+\norm{\eta\nabla u^{\eta,\ve}(\cdot,t)}_{L^{p_m}(U\times[0,T];\mathbb{R}^d)}\right) \\ & \leq C\left(\norm{\nabla \left(u^{\eta,\ve}\right)^{[m]}(\cdot,t)}_{L^{p_m}(U\times [0,T];\mathbb{R}^d)}+\norm{\eta\nabla u^{\eta,\ve}(\cdot,t)}_{L^2(U\times [0,T];\mathbb{R}^d)}\right).\end{aligned}$$
It then follows analogously to Lemma~\ref{lem_bdry} and from Proposition~\ref{char_stable_estimates} and the definition of the entropy defect measure that, for $C=C(m,U,T)>0$, if $m\in(1,\infty)$, for which $p_m\in(1,2)$,
$$\norm{\partial_t\tilde{u}^{\eta,\ve}}_{L^{p_m}\left([0,T];W^{-1,q_m}(U)\right)}\leq C\left(\norm{u_0}^\frac{2}{p_m}_{L^2(U)}+\eta^\frac{1}{2}\norm{u_0}_{L^2(U)}\right),$$
and, if $m\in(0,1]$, for which $p_m=2$,
$$\norm{\partial_t\tilde{u}^{\eta,\ve}}_{L^2\left([0,T];W^{-1,2}(U)\right)}\leq C\left(\norm{u_0}^{\frac{1+m}{2}}_{L^{1+m}(U)}+\norm{u_0}_{L^2(U)}^{\frac{2m}{m+1}}+\eta^\frac{1}{2}\norm{u_0}_{L^2(U)}\right).$$
This completes the proof.  \end{proof}

We now establish the existence of pathwise kinetic solutions for initial data $u_0\in L^2(U)$.  The proof is a consequence of Corollary~\ref{exist_fractional}, Proposition~\ref{exist_time}, and the Aubins-Lions-Simon Lemma \cite{Aubin}, \cite{pLions}, and \cite{Simon}.   We remark that the necessity of an entropy defect in Definition~\ref{sol_def} arises from that fact that, after passing to a subsequence, the gradients
$$\left\{\nabla \left(u^{\eta,\ve}\right)^{\left[\frac{m+1}{2}\right]}\right\}_{\eta,\ve\in(0,1)}$$
will converge only weakly.  Due to the weak lower semi-continuity of the norm, the limit of the parabolic defect measures $\{p^{\eta,\ve}\}_{\eta,\ve\in(0,1)}$ may therefore overestimate the energy of the signed power of the limiting solution.  The total mass of the entropy defect measure quantifies this loss.

\begin{thm}\label{thm_exist}  For each $u_0\in L^2(U)$, there exists a pathwise kinetic solution of \eqref{intro_eq} in the sense of Definition~\ref{sol_def}.  Furthermore, for each $T>0$, the solution satisfies, for $C=C(m,U,T)>0$,
$$\norm{u}^2_{L^\infty([0,T];L^2(U))}+\int_0^T\int_\mathbb{R}\int_U(p+q)\dx\dxi\dr\leq C\norm{u_0}^2_{L^2(U)},$$
and, if $u_0\in L^2_+(U)$, then $u\in L^\infty_{\textrm{loc}}([0,\infty);L^2_+(U))$.
 \end{thm}

\begin{proof}  Let $u_0\in L^2(U)$ be arbitrary.  For each $\eta\in(0,1)$ and $\ve\in(0,1)$, let $u^{\eta,\ve}$ denote the solution of \eqref{char_reg}, and let $\tilde{u}^{\eta,\ve}$ be defined by \eqref{char_sol_transform}.  Proposition~\ref{char_stable_estimates} and the definition of the weight \eqref{c_weight} imply that, for each $T>0$,
\begin{equation}\label{te_0}\{\tilde{u}^{\eta,\ve}\}_{\eta,\ve\in(0,1)}\;\;\textrm{is bounded in}\;\;L^1([0,T];L^1(U)).\end{equation}
Corollary~\ref{exist_fractional} implies that, for each $s\in\left(0,\frac{2}{m+1}\wedge 1\right)$, for each $T>0$,
\begin{equation}\label{te_1}\{\tilde{u}^{\eta,\ve}\}_{\eta,\ve\in(0,1)}\;\;\textrm{is bounded in}\;\;L^1([0,T];W^{s,1}(U)).\end{equation}
Proposition~\ref{exist_time} implies that, for each $T>0$,
\begin{equation}\label{te_2}\{\partial_t\tilde{u}^{\eta,\ve}\}_{\eta,\ve\in(0,1)}\;\;\textrm{is bounded in}\;\;L^1([0,T];W^{-k,1}(U)),\end{equation}
where $W^{-k,1}(U)$ denotes the dual space of $W_0^{k,1}(U)$, and $k>d$ is fixed to guarantee by the Sobolev embedding theorem that
\begin{equation}\label{te_220}W^{k,1}_0(U)\subset L^\infty(U).\end{equation}
Therefore, the boundedness of the domain, the compactness of the embedding $W^{s,1}(U)\hookrightarrow L^1(U)$, the continuity of the embedding $L^1(U)\hookrightarrow W^{-k,1}(U)$ which relies on \eqref{te_220}, and \eqref{te_0}, \eqref{te_1}, and \eqref{te_2} imply with the Aubin-Lions-Simon lemma \cite{Aubin, pLions, Simon} that, for each $T>0$,
$$\{\tilde{u}^{\eta,\ve}\}_{\eta,\ve\in(0,1)}\;\;\textrm{is relatively pre-compact in}\;\;L^1([0,T];L^1(U)).$$
Therefore, after passing to a subsequence $\{(\eta_k,\ve_k)\rightarrow (0,0)\}_{k=1}^\infty$, there exists $\tilde{u}\in L^1_{\textrm{loc}}([0,\infty);L^1(U))$ such that, as $k\rightarrow\infty$, for each $T>0$,
$$\tilde{u}^{\eta_k,\ve_k}\rightarrow \tilde{u}\;\;\textrm{strongly in}\;\;L^1([0,T];L^1(U)).$$
It is then immediate from the definition of the weight \eqref{c_weight}, the weight \eqref{c_w}, and the transformation \eqref{char_sol_transform} that, for
\begin{equation}\label{te_transform}u(x,t):=v^{-1}_{0,t}(x)\tilde{u}(x,t),\end{equation}
as $k\rightarrow\infty$, for each $T>0$,
\begin{equation}\label{te_5}u^{\eta_k,\ve_k}\rightarrow u\;\;\textrm{strongly in}\;\;L^1([0,T];L^1(U)).\end{equation}
Furthermore, as $k\rightarrow\infty$, Proposition~\ref{char_stable_estimates} and the definition of the parabolic defect measure imply that
\begin{equation}\label{te_6}(u^{\eta_k,\ve_k})^{\left[\frac{m+1}{2}\right]}\rightharpoonup u^{\left[\frac{m+1}{2}\right]}\;\;\textrm{weakly in}\;\;L^2([0,T];H^1_0(\mathbb{R}^d)).\end{equation}

Let $\{\chi^{\eta,\ve}\}_{\eta,\ve\in(0,1)}$ denote the kinetic functions of $\{u^{\eta,\ve}\}_{\eta,\ve\in(0,1)}$ defined in \eqref{char_kinetic_new}.  Let $\chi$ denote the kinetic function of $u$ from \eqref{te_transform}.  It follows from \eqref{te_5} and the definition of the kinetic function \eqref{char_kinetic_new} that, as $k\rightarrow\infty$,
\begin{equation}\label{te_7}\chi^{\eta_k,\ve_k}\rightarrow \chi\;\;\textrm{strongly in}\;\;L^1([0,T];L^1(U\times\mathbb{R})).\end{equation}
Proposition~\ref{char_stable_estimates} implies that there exist positive measures $p'$ and $q'$ such that, passing to a further subsequence $k\rightarrow\infty$, for each $T>0$,
\begin{equation}\label{te_60} (p^{\eta_k,\ve_k},q^{\eta_k,\ve_k})\rightharpoonup^*(p',q')\;\;\textrm{weakly in}\;\;\BUC(U\times\mathbb{R}\times[0,T])^*,\end{equation}
where $\BUC$ denotes the space of bounded, uniformly continuous functions.  It follows from the weak lower semicontinuity of the norm and \eqref{te_60} that, in the sense of measures,
$$\delta_0(\xi-u(x,t))\frac{4m}{(m+1)^2}\abs{\nabla u^{\left[\frac{m+1}{2}\right]}}^2\leq q'.$$
We therefore define the parabolic defect measure
$$q=\delta_0(\xi-u(x,t))\frac{4m}{(m+1)^2}\abs{\nabla u^{\left[\frac{m+1}{2}\right]}}^2,$$
and the nonnegative entropy defect measure
\begin{equation}\label{te_8}p=p'+q'-q.\end{equation}
Finally, the continuity of the noise and the regularity of the coefficients together with definitions \eqref{c_weight} and \eqref{c_w} imply that, for each $t>0$, for each $s\in[0,t]$,
\begin{equation}\label{te_9} \lim_{\ve\rightarrow 0}\left(v^\ve_{r,t},\nabla v^\ve_{r,t}, \nabla^2v^\ve_{r,t}\right)=\left(v_{r,t},\nabla v_{r,t}, \nabla^2v_{r,t}\right),\end{equation}
strongly in $L^\infty(U\times[s,t];\mathbb{R}^{1+d+d^2})$, and definitions \eqref{c_backward} and \eqref{c_char} imply that, for each $t>0$, for each $s\in[0,t]$,
\begin{equation}\label{te_10} \lim_{\ve\rightarrow 0}\left(\Pi^{\ve,x,\xi}_{t,t-r},\nabla \Pi^{\ve,x,\xi}_{t,t-r}, \nabla^2\Pi^{\ve,x,\xi}_{t,t-r}, \partial_\xi \Pi^{\ve,x,\xi}_{t,t-r}\right)=\left(\Pi^{x,\xi}_{t,t-r},\nabla \Pi^{x,\xi}_{t,t-r}, \nabla^2\Pi^{x,\xi}_{t,t-r}, \partial_\xi \Pi^{x,\xi}_{t,t-r}\right),\end{equation}
strongly in $L^\infty(U\times\mathbb{R}\times[s,t];\mathbb{R}^{2+d+d^2})$.

The convergence \eqref{te_5} implies that, for a set of Lebesgue measure zero $\mathcal{N}\subset(0,\infty)$, for each $t\in(0,\infty)\setminus\mathcal{N}$,
$$\lim_{k\rightarrow\infty}\norm{u^{\eta_k,\ve_k}(\cdot,t)-u(\cdot,t)}_{L^1(U)}=0.$$
For the kinetic function $\chi$ of $u$, the convergences \eqref{te_6}, \eqref{te_60}, \eqref{te_9} and \eqref{te_10} and the definitions \eqref{te_7} and \eqref{te_8} imply with Proposition~\ref{c_eq} that, for each $s,t\in[0,\infty)\setminus\mathcal{N}$, for every $\rho_0\in \C^\infty_c(U\times \mathbb{R})$, for $\rho_{s,r}$ defined in \eqref{c_flow},
$$\left.\int_\mathbb{R}\int_U\chi\rho_{s,r}v_{s,r}\right|_{r=s}^{t} =  \int_s^t\int_\mathbb{R}\int_Um\abs{\xi}^{m-1}\chi\Delta\left(\rho_{s,r}v_{s,r}\right)-\int_s^t\int_\mathbb{R}\int_U\left(p(x,\xi,r)+q(x,\xi,r)\right)v_{s,r}\partial_\xi\rho_{s,r},$$
where the initial condition is achieved in the sense that, when $s=0$,
$$\int_\mathbb{R}\int_U\chi(x,\xi,0)\rho_{0,0}(x,\xi)v_{0,0}(x) =\int_\mathbb{R}\int_U\chi(x,\xi,0)\rho_0(x,\xi) =\int_\mathbb{R}\int_U\overline{\chi}(u_0(x),\xi)\rho_0(x,\xi).$$
It follows from the convergence \eqref{te_6}, the boundedness of the domain, and the Poincar\'e inequality that, for each $T>0$,
$$u^{\left[\frac{m+1}{2}\right]}\in L^2([0,T];H^1_0(U)).$$
This completes the proof that $u$ is a pathwise kinetic solution with initial data $u_0$.  The estimates are now a consequence of Corollary~\ref{aux_cor} and Proposition~\ref{aux_p}, which complete the proof.  \end{proof}

\begin{appendix}

\section{Boundary Conditions}

In the appendix, for a pathwise kinetic solution $u$, we prove the weak differentiability of $u^{[m]}$.  In particular, this implies that $u^{[m]}$ vanishes on the boundary of $U$ in the sense of a trace.

\begin{lem}\label{lem_bdry}  Let $u_0\in L^2(U)$.  Suppose that $u$ is a pathwise kinetic solution of \eqref{intro_eq} in the sense of Definition~\ref{sol_def} with initial data $u_0$.  Then, for $p_m:=(\frac{m+1}{m}\wedge 2)$, for each $T>0$,
$$u^{[m]}\in L^{p_m}\left([0,T];W^{1,p_m}_0(U)\right)\;\;\textrm{with}\;\;\nabla u^{[m]}=\frac{2m}{m+1}\abs{u}^{\frac{m-1}{2}}\nabla u^{\left[\frac{m+1}{2}\right]}.$$
Furthermore, for $C=C(m,U,T)>0$, if $m\in[1,\infty)$,
$$\norm{u^{[m]}}^{p_m}_{L^{p_m}([0,T];W^{1,p_m}_0(U))}\leq C\norm{u_0}^2_{L^2(U)},$$
and, if $m\in(0,1)$,
$$\norm{u^{[m]}}^2_{L^2([0,T];W^{1,2}_0(U))}\leq C\left(\norm{u_0}^{1+m}_{L^{1+m}(U)}+\norm{u_0}_{L^2(U)}^{\frac{4m}{m+1}}\right).$$
In particular, $u^{[m]}$ has vanishing trace on $\partial U\times(0,\infty)$.
\end{lem}

\begin{proof}  Let $u_0\in L^2(U)$, and suppose that $u$ is a pathwise kinetic solution of \eqref{intro_eq} with initial data $u_0$.  For each $M>0$, let $f_M:\mathbb{R}\rightarrow\mathbb{R}$ denote the unique function satisfying, for each $\xi\in\mathbb{R}$,
$$f'_M(\xi):=\left(\frac{2m}{m+1}\abs{\xi}^{\frac{m-1}{m+1}}\wedge M\right)\;\;\textrm{and}\;\;f_M(0)=0.$$
In particular, as $M\rightarrow\infty$, we have $f_M(\xi)\rightarrow \xi^{\left[\frac{2m}{m+1}\right]}$.

Let $T>0$ be fixed but arbitrary.  Part (i) of Definition~\ref{sol_def} implies that, for each $M>0$,
\begin{equation}\label{bdry_1}f_M\left(u^{\left[\frac{m+1}{2}\right]}\right)\in L^2([0,T];H^1_0(U)),\end{equation}
with
\begin{equation}\label{bdry_100}\nabla f_M\left(u^{\left[\frac{m+1}{2}\right]}\right)=f_M'\left(u^{\left[\frac{m+1}{2}\right]}\right)\nabla u^{\left[\frac{m+1}{2}\right]}.\end{equation}
If $m\in(1,\infty)$, for $p_m:=\frac{m+1}{m}\in(1,2)$, H\"older's inequality, \eqref{bdry_1}, and \eqref{bdry_100} imply that, for each $M>0$ and $t\in[0,T]$, for $C=C(m)>0$,
$$\begin{aligned} \norm{\nabla f_M\left(u^{\left[\frac{m+1}{2}\right]}(\cdot,t)\right)}_{W^{1,p_m}(U;\mathbb{R}^d))} & \leq  \norm{f'_M\left(u^{\left[\frac{m+1}{2}\right]}(\cdot,t)\right)}_{L^{\frac{2(m+1)}{m-1}}(U)}\norm{\nabla u^{\left[\frac{m+1}{2}\right]}(\cdot,t)}_{L^2(U;\mathbb{R}^d)} \\  & \leq  C\norm{u}^{\frac{2}{m-1}}_{L^{m+1}(U)}\norm{\nabla u^{\left[\frac{m+1}{2}\right]}(\cdot,t)}_{L^2(U;\mathbb{R}^d)}.\end{aligned}$$
For $C=C(m,U)>0$, it follows from Lemma~\ref{interpolation} that
$$\norm{\nabla f_M\left(u^{\left[\frac{m+1}{2}\right]}(\cdot,t)\right)}_{W^{1,p_m}(U;\mathbb{R}^d)}\leq C\norm{\nabla u^{\left[\frac{m+1}{2}\right]}(\cdot,t)}_{L^2(U;\mathbb{R}^d)}^{\frac{2m}{m+1}}.$$
In particular, for $C=C(m,U)>0$,
$$\norm{\nabla f_M\left(u^{\left[\frac{m+1}{2}\right]}(\cdot,t)\right)}^{p_m}_{W^{1,p_m}(U;\mathbb{R}^d)}\leq C\norm{\nabla u^{\left[\frac{m+1}{2}\right]}(\cdot,t)}_{L^2(U;\mathbb{R}^d)}^2.$$
Finally, using the definition of the parabolic defect measure and Proposition~\ref{aux_p} with $\delta=1$, for $C=C(m,U,T)>0$,
\begin{equation}\label{bdry_2} \norm{\nabla f_M\left(u^{\left[\frac{m+1}{2}\right]}\right)}^{p_m}_{L^{p_m}([0,T];W^{1,p_m}(U:\mathbb{R}^d))}\leq C\int_0^T\int_\mathbb{R}\int_U q\dx\dxi\dr\leq C\norm{u_0}^2_{L^2(U)}.\end{equation}
Since $p_m\in(1,2)$, H\"older's inequality and \eqref{bdry_1} imply that, for each $M>0$,
$$f_M\left(u^{\left[\frac{m+1}{2}\right]}\right)\in L^{p_m}\left([0,T];W^{1,p_m}_0(U)\right).$$
It follows from \eqref{bdry_1}, \eqref{bdry_100}, \eqref{bdry_2}, and the Poincar\'e inequality that, after passing to the limit $M\rightarrow\infty$,
$$u^{[m]}\in L^{p_m}\left([0,T];W^{1,p_m}_0(U)\right)\;\;\textrm{with}\;\;\nabla u^{[m]}=\frac{2m}{m+1}\abs{u}^{\frac{m-1}{2}}\nabla u^{\left[\frac{m+1}{2}\right]},$$
with, for $C=C(m,U,T)>0$,
$$\norm{u^{[m]}}^{p_m}_{L^{p_m}([0,T];W^{1,p_m}_0(U))}\leq C\norm{u_0}^2_{L^2(U)}.$$

If $m\in(0,1]$, it follows from the definition of the parabolic defect measure and \eqref{bdry_1} that, for $C=C(m)>0$,
$$\begin{aligned}\norm{\nabla f_M\left(u^{\left[\frac{m+1}{2}\right]}\right)}_{L^2([0,T];L^2(U;\mathbb{R}^d))}\leq  & C\int_0^T\int_\mathbb{R}\int_U\abs{f'_M\left(\xi^{\left[\frac{m+1}{2}\right]}\right)}^2q(x,\xi,r) \\ \leq & C\int_0^T\int_\mathbb{R}\int_U\abs{\xi}^{m-1}q(x,\xi,r).\end{aligned}$$
Using Proposition~\ref{aux_p} with $\delta=m$, for $C=C(m,U,T)>0$,
\begin{equation}\label{bdry_4}  \norm{\nabla f_M\left(u^{\left[\frac{m+1}{2}\right]}\right)}^2_{L^2([0,T];L^2(U;\mathbb{R}^d))} \leq C\left(\norm{u_0}^{1+m}_{L^{1+m}(U)}+\norm{u_0}_{L^2(U)}^{\frac{4m}{m+1}}\right).\end{equation}
Therefore, it follows from \eqref{bdry_1}, \eqref{bdry_100}, and \eqref{bdry_4} that, after passing to the limit $M\rightarrow\infty$,
$$u^{[m]}\in L^2\left([0,T];W^{1,2}_0(U)\right)\;\;\textrm{with}\;\;\nabla u^{[m]}=\frac{2m}{m+1}\abs{u}^{\frac{m-1}{2}}\nabla u^{\left[\frac{m+1}{2}\right]},$$
with, for $C=C(m,U,T)>0$,
$$\norm{u^{[m]}}^2_{L^2([0,T];W^{1,2}_0(U))}\leq C\left(\norm{u_0}^{1+m}_{L^{1+m}(U)}+\norm{u_0}_{L^2(U)}^{\frac{4m}{m+1}}\right).$$
This completes the proof.  \end{proof}

\end{appendix}

\section*{Acknowledgements}  The first author was supported by the National Science Foundation Mathematical Sciences Postdoctoral Research Fellowship under Grant Number 1502731.  The second author acknowledges financial support by the DFG through the CRC 1283 ``Taming uncertainty and profiting from randomness and low regularity in analysis, stochastics and their applications.''

\bibliography{MultiplicativeNoise}
\bibliographystyle{plain}

\end{document}